\documentclass[a4paper,reqno, 10pt]{amsart}

\usepackage{amsmath,amssymb,amsfonts,amsthm,mathrsfs}

\usepackage{verbatim,wasysym,cite}

\usepackage[latin1]{inputenc}
\usepackage{microtype}
\usepackage{color,enumitem,graphicx}
\usepackage[colorlinks=true,urlcolor=blue, citecolor=red,linkcolor=blue,
linktocpage,pdfpagelabels, bookmarksnumbered,bookmarksopen]{hyperref}
\usepackage[hyperpageref]{backref}
\usepackage[english]{babel}
\usepackage[symbol]{footmisc}
\renewcommand{\epsilon}{{\varepsilon}}

\numberwithin{equation}{section}
\newtheorem{theorem}{Theorem}[section]
\newtheorem{lemma}[theorem]{Lemma}
\newtheorem{remark}[theorem]{Remark}

\newtheorem{proposition}[theorem]{Proposition}
\newtheorem{corollary}[theorem]{Corollary}

\newcommand{\C}{\mathbb C}

\newcommand{\R}{\mathbb R}
\newcommand{\N}{\mathbb N}

\def\({\left(}
\def\){\right)}
\def\<{\left\langle}
\def\>{\right\rangle}

\def\B{\mathcal B}

\def\K{\mathcal K}
\def\L{\mathcal L}
\def\EE{\mathcal E}

\def\M{\mathcal M}

\DeclareMathOperator{\RE}{Re}
\DeclareMathOperator{\IM}{Im}

\begin{document}

\title[Scattering of the energy-critical NLS with dipolar interaction]{Scattering of the energy-critical NLS with dipolar interaction}

\author[Alex H. Ardila]{Alex H. Ardila}
\address{Universidade Federal de Minas Gerais\\ ICEx-UFMG\\ CEP
  30123-970\\ MG, Brazil} 
\email{ardila@impa.br}

%\author[R\'emi Carles]{R\'emi Carles}
%\address{Univ Rennes, CNRS\\ IRMAR - UMR 6625\\ F-35000 Rennes\\ France}
%\email{Remi.Carles@math.cnrs.fr}

\begin{abstract}
In this paper, we investigate the global well-posedness and $H^{1}$ scattering theory for a 3d energy-critical Schr\"odinger equation 
under the influence of magnetic dipole interaction $\lambda_{1}|u|^{2}u+\lambda_{2}(K\ast|u|^{2})u$, where $K$ is the  dipole-dipole interaction kernel. Our proof of  global well-posedness result is based on the argument of Zhang \cite{Zhang2006}. Moreover, adopting  the induction of energy technique of Killip-Oh-Pocovnicu-Visan \cite{KillipOhPoVi2017}, we obtain a condition for scattering.
\end{abstract}


\maketitle

%\begin{center}
	%\begin{minipage}{10cm}
	%	\small
	%	\tableofcontents
	%\end{minipage}
%\end{center}

\medskip

\section{Introduction}
\label{sec:intro}
The nonlinear Schr\"odinger equations with dipolar interactions have been intensively studied in the last decade. Most works concern local existence \cite{CarlesMarkoSpaber2008}, stability and instability of standing waves \cite{AntonelliSparber2011, LuaAt2019, BellazziniJeanjean2016}, blow-up in finite time and small data scattering \cite{LuoSty2020, CarlesMarkoSpaber2008, BellazziniJeanjean2016, BellazziniForce2020}. Recently, results on scattering for  ``large'' data were obtained \cite{BellazziniForcella2019}.
In this paper, we consider the Cauchy problem for the following energy-critical NLS under the influence of magnetic dipole interaction arising in Bose-Einstein condensation of dipolar quantum gases
\begin{equation}\label{NLS}
\begin{cases} 
(i\partial_{t}+\Delta) u=\lambda_{1}|u|^{2}u+\lambda_{2}(K\ast|u|^{2})u+|u|^{4}u,\\
u(0)=u_{0}\in H^{1}(\R^{3}),
\end{cases} 
\end{equation}
where  $u:\mathbb{R}\times\mathbb{R}^{3}\rightarrow \mathbb{C}$. The physical parameters $\lambda_{1}$,  $\lambda_{2}\in \R$ describe the strength of the dipolar interaction. The dipole interaction kernel is given by
\begin{equation}\label{Dipolo}
K(x)=\frac{1-3\,\text{cos}^{2}(\theta)}{|x|^{3}},
\end{equation}
where $\theta=\theta(x)$ is the angle between the dipole axis $n$ and the vector $x$, i.e. $\text{cos}(\theta)=x\cdot n/|x|$.
For simplicity, without restriction of generality, we fix the dipole axis as the vector $n=(0,0,1)$.  The equations of the form \eqref{NLS} arise in a wide variety of physical models and have been investigated by many authors
\cite{AntonelliSparber2011, BellazziniJeanjean2016, BellazziniForcella2019, CarlesMarkoSpaber2008, Triay2018, LuoSty2020}.
Experimental investigations of the collapse dynamics of \eqref{NLS} in the unstable regime can be found in \cite{Blakie2016,METZ2009}; see
introduction in \cite{LuoSty2020} for more details.

The term ``energy-critical'' refers to the fact that if we ignore the long-range dipolar interaction  
$\lambda_{1}|u|^{2}u+\lambda_{2}(K\ast|u|^{2})u$  (i.e., $\lambda_{1}=\lambda_{2}=0$) in \eqref{NLS}, 
then the family of transformations $u(t,x)=\lambda^{\frac{1}{2}}u(\lambda^{2}t, \lambda^{-1}x)$
preserves both the equation and the energy.  Global well-posedness and scattering theory for the energy critical and mass critical NLS has been intensively studied in the last years; see Colliander-Keel-Staffilani-Takaoka-Tao \cite{TaoKell2008}, Tao-Visan-Zhang \cite{TaoVisanZhang2007}, J. Bourgain \cite{Bourgain1999}, B. Dodson \cite{Dodson2012},  Kenig-Merle\cite{KenigMerle2006}, Killip-Visan\cite{KiiVisan2008}  and references therein for more details.

In this paper, we focus on the case when $\lambda_{1}$ and $\lambda_{2}\in \R$  fulfill the following conditions (the so-called Unstable Regime \cite{CarlesMarkoSpaber2008}):
\begin{equation}\label{UR}
\begin{cases} 
\lambda_{1}-\frac{4\pi}{3}\lambda_{2}<0\quad if\quad \lambda_{2}>0,\\
\lambda_{1}+\frac{8\pi}{3}\lambda_{2}<0\quad if\quad  \lambda_{2}<0.\\
\end{cases} 
\end{equation}
Equation \eqref{NLS} admits the conservation of the energy, defined by
\[\begin{split}
E(u):=\frac{1}{2}\int_{\R^{3}}|\nabla u|^{2}+\frac{\lambda_{1}}{4}|u|^{4}
+\frac{\lambda_{2}}{4}(K\ast|u|^{2})|u|^{2}dx+\frac{1}{6}|u|^{6}dx.
\end{split}\]
for sufficient smooth solutions. We have the 
following result concerning with the well-posedness of the problem \eqref{NLS}.

\begin{theorem}[Global well-posedness]\label{Th1}
Assume that $\lambda_{1}$ and  $\lambda_{2}$ do not vanish simultaneously. Let $u_{0}\in H^{1}(\R^{3})$. Then the initial value  
problem \eqref{NLS} admits  a  unique  global  solution $u\in C(\R, H^{1}(\R^{3}))$ such that $u(0)=u_{0}$. Moreover, we have the conservation of energy, 
mass and momentum: for all $t\in \R$, 
\[
\begin{split}
 M(u(t)):=\| u(t)\|^{2}_{L^{2}}&=M(u_{0}), \,\,\, \mbox{and}\,\,\,  E(u(t))=E(u_{0}),  \\
P(u(t))&:=\int_{\R^{3}}2\IM(\bar{u}\nabla u)dx=P(u_{0}).  
\end{split}
\]
\end{theorem}
In this paper, we consider the scattering theory of the solutions
to \eqref{NLS}. We recall that a solution $u$ to \eqref{NLS} scatters in $H^{1}(\R^{3})$ (both forward and backward time) if there exists a unique scattering state $u_{\pm}\in H^{1}(\R^{3})$ such that
\[\lim_{t\rightarrow\pm\infty}\|u(t)-e^{it\Delta}u_{\pm} \|_{H^{1}}=0,\]
where $e^{it\Delta}$ is the Schr\"odinger group. 

\begin{proposition}[Small data scattering]\label{ThDs}
Let $u_{0}\in H^{1}$ and $u$ denote the global solution to \eqref{NLS} with $u(0)=u_{0}$ given in Theorem \ref{Th1}. Then, there exists $\delta>0$, $\delta$ depending on the $\dot{H}^{1}(\R^{3})$ norm of the initial data $u_{0}$, such that if $\|u_{0}\|_{L^{2}}<\delta$, then $u$ scatters in $H^{1}(\R^{3})$. 
\end{proposition}
In view of Proposition \ref{ThDs}, it is natural to ask under which conditions on the initial date scattering holds. Removing the critical term 
$|u|^{4}u$ one recovers the non-local NLS arising in Bose-Einstein condensation of dipolar quantum gases
\begin{equation}\label{DQG}
(i\partial_{t}+\Delta) u=\lambda_{1}|u|^{2}u+\lambda_{2}(K\ast|u|^{2})u.
\end{equation}
Bellazzini and Forcella \cite{BellazziniForcella2019} have obtained the following 
sufficient condition for scattering of \eqref{DQG}: assume that $\lambda_{1}$, $\lambda_{2}$ belong 
to the unstable regime \eqref{UR} and define the set
\[\begin{split}
\M=\left\{ u\in H^{1}(\R^{3}): E(u_{0})M(u_{0})<E(Q)M(Q)\quad \mbox{and}\right. \\
\left.\|\nabla u_{0}\|_{L^{2}}\|u_{0}\|_{L^{2}}<\|\nabla Q\|_{L^{2}}\|Q\|_{L^{2}}\right\},
\end{split}\]
where $Q$ is a ground state associated with equation \eqref{DQG}. The set $\M$ is invariant 
by the flow of \eqref{DQG} and if $u_{0}\in \M$, then the corresponding solution  to \eqref{NLS} with 
initial data $u_{0}$ is global and scatters. For other results in this direction see also \cite{BellazziniJeanjean2016, LuoSty2020, DinhForceHaja2020, BellazziniForce2020}.

By Proposition \ref{ThDs} we have that solutions of \eqref{NLS} with small initial data scatter in $H^{1}(\R^{3})$. Notice that the equation \eqref{NLS} admits a nonscattering solutions of the form $u(t,x)=e^{i\omega t}Q_{\omega}(x)$, where $Q_{\omega}$ is a nontrivial
solution of the elliptic problem (see Corollary \ref{Thecollo} below)
\begin{equation}\label{Ellip}
-\Delta Q_{\omega}+\lambda_{1}|Q_{\omega}|^{2}Q_{\omega}
+\lambda_{1}(K\ast|Q_{\omega}|^{2})Q_{\omega}+|Q_{\omega}|^{4}Q_{\omega}+\omega Q_{\omega}=0.
\end{equation}
The goal this paper is to find a region $\K$ (see \eqref{DefK} below) of the mass/energy plane such that if the initial data $u_{0}\in H^{1}(\R^{3})$ satisfies  $(M(u_{0}), E(u_{0}))\in \K$, then the corresponding solution with initial data $u_{0}$ scatter in $H^{1}(\R^{3})$.
With this in mind, for $\alpha>0$,  we consider the following sharp $\alpha$-Gagliardo-Nirenberg-H\"older inequality 
\begin{equation}\label{SGNH}
-\lambda_{1}\|f\|^{4}_{L^{4}}-\lambda_{2}\|(K\ast|f|^{2})|f|^{2}\|_{L^{1}}\leq 
C_{\alpha}\|f\|_{L^{2}}\|\nabla f\|^{\frac{3}{1+\alpha}}_{L^{2}}\|f\|^{\frac{3\alpha}{1+\alpha}}_{L^{6}},
\end{equation}
where the sharp constant $C_{\alpha}>0$ is explicitly given by (see Theorem \ref{TheGN} and Corollary \ref{GNcoro})
\begin{equation}\label{apropiate}
C_{\alpha}=\frac{4(1+\alpha)}{3\alpha^{\frac{\alpha}{2(1+\alpha)}}}
{\|Q_{\alpha}\|^{-1}_{L^{2}}\|\nabla Q_{\alpha}\|^{\frac{\alpha-1}{1+\alpha}}_{L^{2}}}.
\end{equation}
We also define the following minimization problem 
\begin{equation}\label{Mef}
\EE(m):=\inf\left\{E(u): u\in H^{1}(\R^{3}),\, M(u)=m \,\, \mbox{and}\,\, I(u)=0 \right\}, 
\end{equation}
where $m\geq0$ and $I$ (virial functional) is given for $v\in H^{1}(\R^{3})$ by
\begin{equation}\label{Virial}
I(v):=\|\nabla v\|^{2}_{L^{2}}+\|v\|^{6}_{L^{6}}+\frac{3}{4}\lambda_{1}\|v\|^{4}_{L^{4}}
+\frac{3}{4}\lambda_{2}\|(K\ast|v|^{2})|v|^{2}\|_{L^{1}}.
\end{equation}
If no function obeys the constrain $\left\{u\in H^{1}(\R^{3}),\, M(u)=m \,\, \mbox{and}\,\, I(u)=0\right\}$, 
then we assume that $\EE(m)=\infty$. Finally, following \cite{KillipOhPoVi2017}, we define the following open region $\K$ of $ \R^{2}$,
\begin{equation}\label{DefK}
\K:= \left\{(m,e): 0<m<M(Q_{1}) \,\, \mbox{and} \,\, 0<e<\EE(m)\right\},
\end{equation}
where is $Q_{1}$ is an optimizer (see \eqref{apropiate}) of the $(\alpha=1)$-Gagliardo-Nirenberg-H\"older inequality \eqref{SGNH}.
\begin{remark}[Description of the region $\K$]\label{Re1}
We set $S(x):=\frac{1}{\sqrt{2}}Q_{1}(\frac{\sqrt{3}}{2}x)$. Notice that $M(S)<M(Q_{1})$, $I(S)=0$ and $E(S)>0$ (see Section \ref{SC1} for more details). In Theorem \ref{SK} we will show that
the function $\EE: (0, M(Q_{1})]\rightarrow \R$ given by  \eqref{Mef}  satisfies: $\EE(m)=\infty$ on $(0, M(S))$, $\EE(m)\in (0, E(S)]$ on $[M(S), M(Q_{1}))$ and  $\EE(m)=0$ when $m=M(Q_{1})$. 
Moreover, $\EE$ is strictly decreasing  on $[M(S), M(Q_{1})]$; see Figure \ref{Figu1} 
for an illustration of the shape of the function $\EE$.
\end{remark}
%%%%%%%%%%%%%%%%%%%%%%%%%%%%%%%%%%%%%%%%%%%%%%%%%%%%%%%%%%%%%%%%%%%%%%%%%%%%%%%%%%%%%%5
\begin{figure}[h]\label{Figu1}
\caption{Mass/Energy Plane}
\centering
\includegraphics[width=\textwidth]{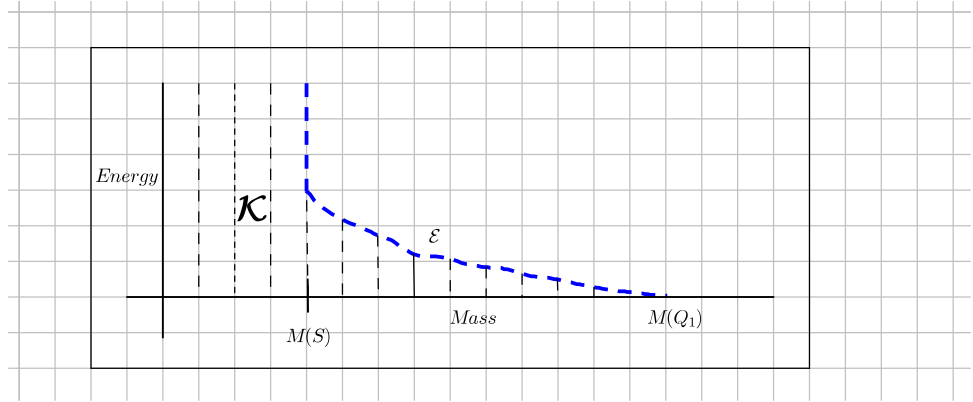}
\end{figure}
%%%%%%%%%%%%%%%%%%%%%%%%%%%%%%%%%%%%%%%%%%%%%%%%%%%%%%%%%%%%%%%%%%%%%%%%%%%%%%%%%%%%%%%5
Since $\K$ contains a neighborhood of $(0,0)$, and solutions of equation \eqref{NLS} with 
small initial values scatter (see Proposition \ref{ThDs}), a question that naturally arises is whether or not all  solutions of \eqref{NLS} such that
$(M(u),E(u))\in \K$ also scatter. In this paper, we give a positive answer to this question. 
\begin{theorem}\label{TheS}
Let $\lambda_{1}$, $\lambda_{2}$ belong to the unstable regime \eqref{UR} and let $\K$ be defined by \eqref{DefK}.
Given any $u_{0}\in H^{1}(\R^{3})$ such that $(M(u_{0}),E(u_{0}))\in \K$, it follows that the corresponding global solution $u$
of \eqref{NLS} scatters in $ H^{1}(\R^{3})$. 
\end{theorem}
Notice that by Theorem \ref{TheS} and Remark \ref{Re1} we infer that if the initial data $u_{0}\in H^{1}(\R^{3})$ satisfies 
$\|u_{0}\|_{L_{x}^{2}}<\|S\|_{L_{x}^{2}}$, then the corresponding solution $u(t)$ to \eqref{NLS} scatter in $H^{1}(\R^{3})$;
compare with Proposition \ref{ThDs}.

Our proof scattering result is based on the argument of Killip-Oh-Pocovnicu-Visan \cite{KillipOhPoVi2017}, which treated the cubic-quintic problem in three dimensions.  We remark that the so-called unstable regime in the
original dipole model can be thought of as a focusing cubic nonlinearity (see Remark 2.3 in \cite{AntonelliSparber2011}), and therefore it is natural to use Killip-Oh-Pocovnicu-Visan's \cite{KillipOhPoVi2017} approach to study the equation \eqref{NLS}. In particular, we adopt the concentration compactness approach to induction on energy. However, in the present situation new technical problems appear related especially to the presence of the dipole interaction kernel that makes the analysis more delicate.
\subsection*{Stable regime}  We say that $\lambda_{1}$ and $\lambda_{2}\in \R$ belong to the stable regime if \eqref{UR} 
is not satisfied, i.e.,
\begin{equation}\label{ER1}
\begin{cases} 
\lambda_{1}-\frac{4\pi}{3}\lambda_{2}\geq 0\quad if\quad \lambda_{2}>0,\\
\lambda_{1}+\frac{8\pi}{3}\lambda_{2}\geq 0\quad if\quad  \lambda_{2}<0.\\
\end{cases} 
\end{equation}
Notice that in this case we have that $-\lambda_{1}\|\varphi\|^{4}_{L^{4}}-\lambda_{2}\|(K\ast|\varphi|^{2})|\varphi|^{2}\|_{L^{1}}\leq 0$
for every $\varphi\in H^{1}(\R^{3})$. In particular, from Lemma \ref{Phozine} (see also \cite[Theorem 3.3 (1)]{LuoSty2020}) below we infer that  no solutions to stationary problem \eqref{Ellip} exists  in $H^{1}(\R^{3})$. Moreover, we can use the same argument as in Theorem \ref{TheS}, with some natural modifications (see \cite[Section 7]{DuyHolmerRoude2008} and \cite{Lafontaine2016} for more details), to establish the following result.
\begin{corollary}
Let $\lambda_{1}$, $\lambda_{2}$ belong to the stable regime \eqref{ER1}. Then every solution of \eqref{NLS} scatters in $H^{1}(\R^{3})$.
\end{corollary}
Outline of the paper. At the end of Introduction we fix notation to be used throughout the rest of paper. In Section \ref{S:preli}
we collect some useful results, including Stricharzt estimates. In Section \ref{S:es}, we prove the Theorem \ref{Th1} and Proposition \ref{ThDs}. In Section \ref{S:2} we consider the  sharp $\alpha$-Gagliardo-Nirenberg-H\"older inequality \eqref{SGNH}. 
In Section \ref{SC1}, we give the description of the set $\K$. Section \ref{S:Tes} is devoted to a stability result for \eqref{NLS} and a profile decomposition property. Section \ref{Sps} is devoted to the construction of a critical solution. Finally, in Section \ref{Su8} we prove the scattering result by a rigidity argument (Theorem \ref{TheS}).

\subsection*{Notations}
We write $A\lesssim B$ or $B\gtrsim A$  to denote $A\leq CB$ for some $C>0$. If  $A \lesssim B \lesssim A$ we write $A\sim B$.
For a function $u:I\times \R^{3}\rightarrow \C$, $I\subset \R$, we write
\[ \|  u \|_{L_{t}^{q}L^{r}_{x}(I\times \R^{3})}=\|  \|u(t) \|_{L^{r}_{x}(\R^{3})}  \|_{L^{q}_{t}(I)}
 \]
with $1\leq q\leq r\leq\infty$. Moreover, the Fourier transform on $\R^{3}$ is defined by 
\[\hat{f}(\xi)=\int_{\R^{3}}e^{-ix\cdot\xi}f(x)dx.
\]
We say that $(q,r)$ is admissible if $2\leq q,r\leq\infty$ and $\frac{2}{q}+\frac{3}{r}=\frac{3}{2}$. Also, we define
\[\| u  \|_{S^{0}(I)}:= \sup\left\{ \|  u \|_{L_{t}^{q}L^{r}_{x}(I\times \R^{3})}: \quad \text{$(q,r)$ is admissible}\right\}.
\]
We use the ``Japanese bracket'' $\<\nabla\>=(1-\Delta)^{1/2}$ and we define the Sobolev norms 
\[\|u\|_{H^{s,r}(\R^{3})}:=\|\<\nabla\>^{s} u\|_{L^{r}_{x}(\R^{3})}.\]
Given $p$, we let $p'$ denote the conjugate of $p$.

\section{Preliminaries}\label{S:preli}  
We have the following global-in-time Strichartz estimates.
\begin{lemma}[Strichartz estimates]
Let $(q,r)$ be an admissible pair. Then the solution $u$ of $(i\partial_{t}+\Delta) u=F$ with initial data $u_{0}$
obeys
\[\| u  \|_{L_{t}^{q}L^{r}_{x}(I\times \R^{3})} \lesssim \| u_{0}\|_{L^{2}_{x}(\R^{3})}+ \| F\|_{L_{t}^{\tilde{q}'}L^{\tilde{r}'}_{x}(I\times \R^{3})}, \]
where $2\leq \tilde{q},\tilde{r}\leq\infty$ with $\frac{2}{\tilde{q}}+\frac{3}{\tilde{r}}=\frac{3}{2}$ and for some interval $I\subset \R$. 
\end{lemma}

As mentioned in the introduction, we can assume that the applied magnetic field is parallel to the $x_{3}$-axis. In this case 
$n=(0,0,1)$ and we have that (see \eqref{Dipolo})
\[K(x)=\frac{x^{2}_{1}+x^{2}_{2}-2x^{2}_{3}}{|x|^{5}}.\]
 We have the following result (see \cite{CarlesMarkoSpaber2008} for more details).
\begin{lemma}\label{LLk}
The operator $K: u\mapsto K\ast u$ can be extended as a continuous operator on $L^{p}(\R^{3})$ for $1<p<\infty$. In addition,
the Fourier transform of $K$ is given by
\[\hat{K}(\xi)= \frac{4\pi}{3} \(\frac{2 \xi^{2}_{3}- \xi^{1}_{2}-\zeta^{2}_{2}}{|\xi|^{2}}   \).\]
In particular, $\hat{K}(\xi)\in [-\frac{4\pi}{3},\frac{8\pi}{3}]$.
\end{lemma}

We will need moreover the following result (see \cite[Lemma 8.1.]{LuoSty2020}).
\begin{lemma}\label{Lne}
Let $\lambda_{1}$, $\lambda_{2}$ belong to the unstable regime \eqref{UR}. Then for each $m>0$ there exists some $f\in H^{1}(\R^{3})$
with $\| f \|_{L^{2}}=m$ such that $N(f)>0$, where 
\[N(f):=-\lambda_{1}\|f\|^{4}_{L^{4}}-\lambda_{2}\|(K\ast|f|^{2})|f|^{2}\|_{L^{1}}.\]
\end{lemma}

Finally, we introduce the energy-critical NLS in $\R^{3}$,
\begin{equation}\label{ECgNLS}
\begin{cases} 
(i\partial_{t}+\Delta) u=|u|^{4}u,\\
u(0)=u_{0}\in \dot{H}_{x}^{1}(\R^{3}),
\end{cases} 
\end{equation}
which play an important role in the scattering theory of \eqref{NLS}. We have the following result.
\begin{theorem}[Global well-posedness]\label{WGPQN}
For any $u_{0}\in \dot{H}_{x}^{1}(\R^{3})$ there exists a unique global solution $u\in C(\R, \dot{H}_{x}^{1}(\R^{3}))$ to \eqref{ECgNLS}.
Furthermore, we have the following space-time bound
\[
\| u  \|^{10}_{L^{10}_{t,x}(\R\times\R^{3})}\leq C({\|u_{0}\|_{\dot{H}_{x}^{1}}}).
\]
\end{theorem}
For the proof of Theorem \ref{WGPQN} in the radial setting see Bourgain \cite{Bourgain1999}, and in its full generality
see Colliander, Keel, Staffilani, Takaoka and Tao \cite{TaoKell2008}; see also \cite{ KiiVisan2008}.

\section{Well-posedness and small data scattering}\label{S:es} 
In this section we study the local theory of \eqref{NLS}.  It plays a important role in  scattering theory. The idea is originally due to Zhang \cite{Zhang2006}. First we need the following result, which gives local existence,  uniqueness and 
conservation of mass and energy for any initial data in $H^{1}(\R^{3})$.

\begin{proposition} \label{LWP}
Assume that $\lambda_{1}$ and  $\lambda_{2}$ do not vanish simultaneously. For any $u_{0}\in H^{1}(\R^{3})$, there exists a unique solution $u\in C(I, H^{1}(\R^{3}))$ to the Cauchy problem \eqref{NLS} on some interval $I:=(-T_{{min}}, T_{{max}})\ni 0$. Furthermore,\\
{ \rm (i)} There is conservation of mass and energy.\\
{ \rm (ii)} $u\in L_{t}^{q}L_{x}^{r}(K\times \R^{3})$  for every  compact time interval $K\subset I$ and for any admissible pair $(q,r)$. \\
{ \rm (iii)} If $T_{{max}}$ is finite, then 
\[\lim_{t\rightarrow T_{{max}}}\|u(t)\|_{H^{1}_{x}}=\infty\quad \text{or} \quad \sup_{(q,r)-\text{admissible}}\| u\|_{L_{t}^{q}H^{1,r}_{x}((0,T_{{max}})\times \R^{3})}=\infty,\]
for every admissible par $(q,r)$ with $2<r<6$.
An analogous statement holds when $T_{{min}}$ is finite.\\
{ \rm (iv)} The solution $u$ depends  continuously on the initial data $u_{0}$.
\end{proposition}
\begin{proof}
We can prove the proposition in the same way as in \cite[Section 4.5]{CB}. So we omit the details.
\end{proof}

\begin{lemma} \label{GWP1}
The solution $u$ to \eqref{NLS} given in Proposition \ref{LWP} is global, i.e. $I=\R$.
\end{lemma}
\begin{proof}
\textsl{Step 1.} \textsl{Kinetic energy control}. By Plancherel identity, the energy functional can be rewritten as
\[E(u)=\frac{1}{2}\int_{\R^{3}}|\nabla u|^{2}dx+\frac{1}{4}\frac{1}{(2\pi)^{3}}\int_{\R^{3}}(\lambda_{1}+\lambda_{2}\widehat{K}(\xi))|\widehat{u^{2}}|^{2}d\xi+\frac{1}{6}\int_{\R^{3}}|u|^{6}dx.\]
Since $\widehat{K}(\xi)\in \left[-\frac{4}{3}\pi, \frac{8}{3}\pi\right]$, we infer that there exists $\beta>0$ such that
\[ -\lambda_{1}\|u\|^{4}_{L^{4}}-\lambda_{2}\|(u\ast|u|^{2})|u|^{2}\|_{L^{1}}\leq \beta\|u\|^{4}_{L^{4}},\]
and so we have
\begin{align*}
E(u)+M(u)\gtrsim \|\nabla u\|^{2}_{L^{2}}+\frac{1}{3}\int_{\R^{3}}|u|^{2}\(|u|^{2}-\frac{3}{4}\beta \)^{2}dx.
\end{align*}
This implies that 
\begin{equation}\label{UT}
\sup_{t\in I}[\|\nabla u(t)\|^{2}_{L^{2}}+\| u(t)\|^{2}_{L^{2}}] \lesssim E(u_{0})+M(u_{0}).
\end{equation}
\textsl{Step 2.} \textsl{Local space-time bound}.  We consider the defocusing quintic Schr\"odinger equation
\begin{equation}\label{CNLS}
\begin{cases} 
(i\partial_{t}+\Delta) v=|v|^{4}v,\\
v(0)=u_{0}\in \dot{H}^{1}(\R^{3}).
\end{cases} 
\end{equation}
By Theorem \ref{WGPQN} we have that there exists a unique global solution $v\in C(\R, \dot{H}_{x}^{1})$ to \eqref{CNLS} which satisfies 
\begin{equation}\label{Uc}
 \|v\|_{L_{t}^{q}L_{x}^{r}(\R\times \R^{3})} + \|v\|_{L_{t}^{q}\dot{H}_{x}^{1,r}(\R\times \R^{3})} 
\leq  C(\|u_{0}\|_{\dot{H}_{x}^{1}}, \|u_{0}\|_{L^{2}}),
 \end{equation}
for every admissible pair $(q,r)$. We set $w:=u-v$, where $u$ is the local solution to \eqref{NLS} given in Proposition \ref{LWP} and $v$ is the 
global solution of the problem \eqref{CNLS}. On taking the difference of the two equations we infer that $w$ 
satisfies the Cauchy problem:
\begin{equation}\label{0IP}
\begin{split} 
(i\partial_{t}+\Delta) w=\lambda_{1}|w+u|^{2}(w+u)+\lambda_{2}(K\ast|w+u|^{2})(w+u)\\
+|w+u|^{4}(w+u)-|v|^{4}v,\\
\end{split} 
\end{equation}
 with initial condition $w(0)=0$. For a time slab $I\subset \R$ we define the spaces
\[\begin{split}
\dot{B}^{0}_{I}=&L_{t}^{\frac{10}{3}}L_{x}^{\frac{10}{3}}(I\times \R^{3})\cap L_{t}^{10}L_{x}^{\frac{30}{13}}(I\times \R^{3})\cap L_{t}^{8}L_{x}^{\frac{12}{5}}(I\times \R^{3}),\\
&\dot{B}^{1}_{I}=\left\{u; \nabla u\in \dot{B}^{0}_{I}  \right\}, \quad {B}^{1}_{I}=\dot{B}^{0}_{I}\cap \dot{B}^{1}_{I}.
\end{split}\]
From Lemma 1.7 in \cite{Zhang2006} we have
\[\begin{split}
&\|\nabla^{s}(|u|^{4}u)\|_{L_{t}^{\frac{10}{7}}L_{x}^{\frac{10}{7}}(I\times \R^{3})}\lesssim \|u\|^{4}_{\dot{B}^{1}_{I}}\|u\|_{\dot{B}^{s}_{I}}\\
&\|\nabla^{s}(|u|^{2}u)\|_{L_{t}^{\frac{8}{7}}L_{x}^{\frac{12}{7}}(I\times \R^{3})}\lesssim |I|^{\frac{1}{2}}\|u\|^{2}_{\dot{B}^{1}_{I}}\|u\|_{\dot{B}^{s}_{I}},
\end{split}\]
for $s=0$, $1$. On the other hand, using the fact that $\frac{1}{12}=\frac{1}{12}+\frac{1}{2}$, $\frac{1}{2}=\frac{1}{12}+\frac{5}{12}$, 
by H\"older inequality and Sobolev embedding we obtain for $s=0$, $1$
\[\begin{split}
\|\nabla^{s}[(K\ast|u|^{2})u]\|_{L_{t}^{\frac{8}{7}}L_{x}^{\frac{12}{7}}(I\times \R^{3})}\lesssim |I|^{\frac{1}{2}}\|u\|^{2}_{\dot{B}^{1}_{I}}\|u\|_{\dot{B}^{s}_{I}}.
\end{split}\]
The proof of Lemma \ref{GWP1} now follows the same lines as in \cite[Section 2.3]{Zhang2006}. Indeed, by the inequalities above, \eqref{Uc} and using the  inductive argument developed in \cite[Pag 431]{Zhang2006}  we have that the initial value problem \eqref{0IP} has a unique solution $w$ on an interval $[0, T]$ ($T=T(\|u_{0}\|_{H^{1}})>0$) such that $\|w\|_{B^{1}_{[0, T]}} \lesssim C$. Since, by construction,  $u=v+w$ on $[0, T]$, we obtain a  unique solution $u$ of \eqref{NLS} on an interval $[0, T]$ such that
\begin{equation}\label{LCU}
\|u\|_{B^{1}_{[0, T]}}\leq \|w\|_{B^{1}_{[0, T]}}+\|v\|_{B^{1}_{[0, T]}} \lesssim_{u_{0}} C
\end{equation}
Therefore, since the equation \eqref{NLS} is time translation invariant, by \eqref{UT} and \eqref{LCU} we infer that the solution is global. 
This completes the proof of Lemma.
\end{proof}

\begin{proof}[{Proof of Theorem \ref{Th1}}]
Theorem \ref{Th1} follows immediately from Proposition \ref{LWP} and Lemma \ref{GWP1}.
\end{proof}

In the following result we show that the solution scatters in $H^{1}(\R^{3})$ when the initial data is small in $L^{2}(\R^{3})$.
\begin{proposition}\label{Scd}
Let $u_{0}\in H^{1}(\R^{3})$. Then there exists a constant $\delta>0$ (depending on the $\dot{H}^{1}(\R^{3})$-norm of $u_{0}$) such that if, $\|u_{0}\|_{L^{2}}\leq \delta$, then the corresponding solution $u$ of \eqref{NLS} satisfies the global space time bound $\|u\|_{L^{q}_{t}H^{1, r}_{x}(\R\times \R^{3})} \lesssim 1$ for any admissible pair $(p,r)$. In particular, $u$ scatters in $H^{1}_{x}$.
\end{proposition}
\begin{proof}
\textsl{Step 1.} \textsl{Global space-time bound.}   Following \cite[Subsection 3.3]{Zhang2006}, the 
idea is to approximate \eqref{NLS} by \eqref{CNLS} globally in the time. 
With this in mind, we define the spaces
\[\dot{Z}^{0}_{I}=L_{t}^{\frac{10}{3}}L_{x}^{\frac{10}{3}}(I\times \R^{3})\cap L_{t}^{10}L_{x}^{\frac{30}{13}}(I\times \R^{3}),\quad
\dot{Z}^{1}_{I}=\left\{u; \nabla u\in \dot{Z}^{0}_{I} \right\}, \quad  {Z}^{1}_{I}=\dot{Z}^{0}_{I}\cap \dot{Z}^{1}_{I},\]
for a time slab $I\subset \R$. 
Let $v$ the global solution of the defocusing quintic Schr\"odinger equation \eqref{CNLS} with $v(0)=u_{0}$. Again, as in Lemma \ref{GWP1} above, we set $w:=u-v$, where $u$ is the global solution to \eqref{NLS} with initial data $u_{0}$. Then $w$ satisfies the Cauchy problem 
\begin{equation}\label{0IPN}
\begin{split} 
(i\partial_{t}+\Delta) w=\lambda_{1}|w+v|^{2}(w+v)+\lambda_{2}(K\ast|w+v|^{2})(w+v)\\
+|w+v|^{4}(w+v)-|v|^{4}v,\\
\end{split} 
\end{equation}
with initial data $w(0)=0$. Making using of H\"older inequality we have (see \cite[Pag 437]{Zhang2006})
\begin{equation}\label{Pgd}
\begin{split}
\||w+v|^{2}(w+v)\|_{L_{t}^{\frac{10}{7}}H_{x}^{1, \frac{10}{7}}(I\times \R^{3})}\lesssim \|w\|^{2}_{{Z}^{1}_{I}}+
\|w\|_{{Z}^{1}_{I}}\|v\|_{{\dot{Z}}^{1}_{I}}\\
+\|v\|_{{\dot{Z}}^{1}_{I}}\|v\|_{{\dot{Z}}^{0}_{I}}\|w\|_{{Z}^{1}_{I}}+\|v\|^{2}_{{\dot{Z}}^{1}_{I}}\|v\|_{{\dot{Z}}^{0}_{I}}.
\end{split}
\end{equation}
Now, H\"older inequality implies that
\begin{equation}\label{Pge}
\begin{split}
\||w+v|^{4}(w+v)-|v|^{4}v\|_{L_{t}^{\frac{10}{7}}H_{x}^{1, \frac{10}{7}}(I\times \R^{3})}\lesssim 
\sum^{4}_{i=1}\|w\|^{5-i}_{{Z}^{1}_{I}}\|v\|^{i}_{{\dot{Z}}^{1}_{I}}.
\end{split}
\end{equation}
Moreover, since the operator $f\mapsto K\ast f$ is continuous on $L^{p}(\R^{3})$ 
for each $1<p<\infty$ (see Lemma \ref{LLk}),  we obtain
\[
\begin{split}
\|(K\ast|w+v|^{2})(w+v)\|_{H_{x}^{1, \frac{10}{7}}}\lesssim 
\|(K\ast (1+|\nabla|)|w+v|^{2})(w+v)\|_{L_{x}^{2}}\|w+v\|_{L^{5}_{x}}\\
+\|K\ast|w+v|^{2}\|_{L^{\frac{5}{2}}_{x}}+\|(1+|\nabla|)\, (w+v)\|_{L^{\frac{10}{3}}_{x}}\\
\lesssim \|w+v\|^{2}_{L^{5}_{x}}\|(1+|\nabla|)\, (w+v)\|_{L^{\frac{10}{3}}_{x}}.
\end{split}
\]
Therefore,
\[
\begin{split}
\|(K\ast|w+v|^{2})(w+v)\|_{L^{\frac{10}{7}}_{t}H_{x}^{1, \frac{10}{7}}}\lesssim 
 \|w+v\|^{2}_{L^{5}_{t}L^{5}_{x}}\|w+v\|_{L^{\frac{10}{3}}_{t}H^{1,\frac{10}{3}}_{x}},
\end{split}
\]
and thanks to inequality $\|f\|^{2}_{L^{5}_{t}L^{5}_{x}}\leq \|f\|_{{\dot{Z}}^{1}_{I}}\|f\|_{{\dot{Z}}^{0}_{I}}$
 we deduce that
\begin{equation}\label{Cic}
\begin{split}
\|(K\ast|w+v|^{2})(w+v)\|_{L^{\frac{10}{7}}_{t}H_{x}^{1, \frac{10}{7}}}\lesssim \|w\|^{2}_{{Z}^{1}_{I}}+
\|w\|_{{Z}^{1}_{I}}\|v\|_{{\dot{Z}}^{1}_{I}}\\
+\|v\|_{{\dot{Z}}^{1}_{I}}\|v\|_{{\dot{Z}}^{0}_{I}}\|w\|_{{Z}^{1}_{I}}+\|v\|^{2}_{{\dot{Z}}^{1}_{I}}\|v\|_{{\dot{Z}}^{0}_{I}}.
\end{split}
\end{equation}
Combining the estimates \eqref{Pgd}, \eqref{Pge} and \eqref{Cic} and using the inductive argument 
developed in \cite[Pag 438]{Zhang2006} we can show that there exists $\delta:=\delta(\dot{H}^{1}(\R^{3}))>0$ such that if $\|u_{0}\|_{L^{2}}\leq \delta$, then the initial value problem \eqref{0IPN} has a unique solution in $\R$ such that $\|w\|_{{Z}^{1}_{\R}}\leq C(\|u_{0}\|_{H^{1}})$. From \eqref{Uc}, one gets
\[ \|u\|_{{Z}^{1}_{\R}}\leq  \|w\|_{{Z}^{1}_{\R}}+\|v\|_{{Z}^{1}_{\R}}\leq C(\|u_{0}\|_{H^{1}}).\]
Finally, making using of Strichartz estimates, we have
\[\|u\|_{L^{q}_{t}H^{1,r}_{x}(\R\times \R^{3})} \leq C(\|u_{0}\|_{H^{1}}),\]
for any admissible pair $(p,r)$.\\

\textsl{Step 2.} \textsl{Scattering.} According the Duhamel formula, for 
$F(u)=\lambda_{1}|u|^{2}u+\lambda_{2}(K\ast|u|^{2})u+|u|^{4}u$, we define
\[u_{+}(t)=u_{0}-i\int^{t}_{0}e^{i(t-\tau)}F(u(\tau))d\tau.\]
This is well defined in $H^{1}(\R^{3})$, since by H\"older (see \eqref{Pgd}, \eqref{Pge} and \eqref{Cic}) and Strichartz estimates we get
\[\begin{split}
\|u_{+}(t_{1})-u_{+}(t_{2})\|_{H^{1}}
&\lesssim \|u_{+}(t_{1})-u_{+}(t_{2})\|_{L^{\frac{10}{7}}_{t}H_{x}^{1, \frac{10}{7}}([t_{1},t_{2}]\times \R^{3})}\\
&\lesssim \|u\|^{3}_{{Z}^{1}_{[t_{1},t_{2}]}}+\|u\|^{5}_{{Z}^{1}_{[t_{1},t_{2}]}}\rightarrow 0,
\end{split}\]
as $t_{1}$, $t_{2}\rightarrow\infty$.
Analogously we have
\[\begin{split}
\|u(t)-e^{it}u_{+}\|_{H^{1}}\lesssim \|u\|^{3}_{{Z}^{1}_{[t_{1},t_{2}]}}+\|u\|^{5}_{{Z}^{1}_{[t_{1},t_{2}]}},
\end{split}\]
which is converging to $0$ as $t\rightarrow \infty$. This completes the proof of proposition.
\end{proof}

\begin{proof}[{Proof of Proposition \ref{ThDs}}]
Proposition \ref{ThDs}  is an immediately consequence of Proposition \ref{Scd}.
\end{proof}

\section{The sharp Gagliardo-Nirenberg-H\"older inequality}\label{S:2} 

In section we dicuss the sharp $\alpha$-Gagliardo-Nirenberg-H\"older inequality \eqref{SGNH}.
We begin with the following lemma.
\begin{lemma}\label{Phozine}
If $\varphi$ satisfies the equation \eqref{Ellip}, then the following identities hold:
\begin{align} \label{Pho1}
&\|\nabla \varphi\|^{2}_{L^{2}}+\|\varphi\|^{6}_{L^{6}}-N(\varphi)+\omega\|\varphi\|^{2}_{L^{2}}=0\\
&\frac{1}{6}\|\nabla \varphi\|^{2}_{L^{2}}+\frac{1}{6}\|\varphi\|^{6}_{L^{6}}-\frac{1}{4}N(\varphi)
+\frac{\omega}{2}\|\varphi\|^{2}_{L^{2}}=0. \label{Pho23},
\end{align}
where $N(\varphi):=-\lambda_{1}\|\varphi\|^{4}_{L^{4}}-\lambda_{2}\|(K\ast|\varphi|^{2})|\varphi|^{2}\|_{L^{1}}$. In particular, if $\varphi\neq0$, then $\omega>0$ and $N(\varphi)>0$.
\end{lemma}
\begin{proof}
To obtain \eqref{Pho1}, we multiply \eqref{Ellip} by $\bar{\varphi}$ and integrate over $\R^{3}$. Similarly, 
multiplying \eqref{Ellip} by $x\cdot\nabla \varphi$ and integrate we get \eqref{Pho23}. A rigorous proof  
can be found in \cite[Section 2.1]{LIONSBere1983} (see also \cite[Lemma 2.2]{AntonelliSparber2011}).
On the other hand, combining the identities \eqref{Pho1} and \eqref{Pho23} we obtain
\begin{equation}\label{N11}
N(\varphi)=4\omega \|\varphi\|^{2}_{L^{2}}, \quad  \|\nabla \varphi\|^{2}_{L^{2}}+\|\varphi\|^{6}_{L^{6}}=3\omega \|\varphi\|^{2}_{L^{2}}.
\end{equation}
Thus we infer that $\omega>0$. Finally, in view of \eqref{N11} we see that $N(\varphi)>0$ when $\varphi\neq0$. 
\end{proof}

Through this section we assume that $\lambda_{1}$, $\lambda_{2}$ belong to the unstable regime \eqref{UR}.

\begin{theorem}[$\alpha$-Gagliardo-Nirenberg-H\"older inequality]\label{TheGN}
Let $0<\alpha<\infty$, Then the infimum
\begin{equation}\label{GNI}
C^{-1}_{\alpha}:=\inf_{ f\in \B}\frac{\|f\|_{L^{2}}\|\nabla f\|^{\frac{3}{1+\alpha}}_{L^{2}}\|f\|^{\frac{3\alpha}{1+\alpha}}_{L^{6}}}{-\lambda_{1}\|f\|^{4}_{L^{4}}-\lambda_{2}\|(K\ast|f|^{2})|f|^{2}\|_{L^{1}}},
\end{equation}
where $\B:=\left\{f\in H^{1}(\R^{3}): -\lambda_{1}\|f\|^{4}_{L^{4}}-\lambda_{2}\|(K\ast|f|^{2})|f|^{2}\|_{L^{1}}>0 \right\}$ is  attained. Furthermore,  any  minimizing  sequence of problem \eqref{GNI} is  relatively  compact  in $H^{1}(\R^{3})$ up to dilations, multiplication by constants and translations.\\
(Note that the inequality \eqref{SGNH} holds trivially if $-\lambda_{1}\|f\|^{4}_{L^{4}}-\lambda_{2}\|(K\ast|f|^{2})|f|^{2}\|_{L^{1}}\leq 0$).
\end{theorem}
\begin{proof}
First we show that $0<C_{\alpha}<\infty$. Indeed, we set
\begin{equation}\label{Nd}
N(f):=-\lambda_{1}\|f\|^{4}_{L^{4}}-\lambda_{2}\|(K\ast|f|^{2})|f|^{2}\|_{L^{1}}.
\end{equation}
Notice that Lemma \ref{Lne} ensure that the set $\B$ is not empty. Thus, $C_{\alpha}>0$. Also note
\[ 
N(f)\lesssim \|f\|^{4}_{L^{4}}
\lesssim \|f\|_{L^{2}}\|f\|^{3}_{L^{6}},
\]
then by the Sobovev embedding $\dot{H}^{1}(\R^{3})\hookrightarrow L^{6}(\R^{3})$ we infer
\[\begin{split}
\frac{\|f\|_{L^{2}}\|\nabla f\|^{\frac{3}{1+\alpha}}_{L^{2}}\|f\|^{\frac{3\alpha}{1+\alpha}}_{L^{6}}}{N(f)}
 \gtrsim \frac{\|f\|_{L^{2}}\|f\|^{3}_{L^{6}}}{N(f)}
\gtrsim 1.
\end{split}\]
This implies that $C_{\alpha}<\infty$.
Next we define the Weinstein functional
\[W(f):= \frac{\|f\|_{L^{2}}\|\nabla f\|^{\frac{3}{1+\alpha}}_{L^{2}}\|f\|_{L^{6}}^{\frac{3\alpha}{1+\alpha}}}{N(f)}.\]
Let $\left\{f_{n}\right\}_{n\in \N}$ be a minimizing sequence for the infimum $C_{\alpha}$. By exploiting 
the fact that $W(f_{q,s})=W(f)$, where $f_{q,s}(x):=qf(s x)$, we can rescale the sequence $\left\{f_{n}\right\}_{n\in \N}$
such that $\|f_{n}\|_{L^{2}}=1$ and  $\|\nabla f_{n}\|_{L^{2}}=1$.   Thus, $\left\{f_{n}\right\}_{n\in \N}$ is a bounded sequence in
$H^{1}(\R^{3})$. Since
\[W(f_{n})=\frac{\|f_{n}\|^{\frac{\alpha}{\alpha+1}}_{L^{2}}\|f_{n}\|^{\frac{3\alpha}{\alpha+1}}_{L^{6}}}{N(f_{n})}
\gtrsim
\frac{\|f_{n}\|^{\frac{4\alpha}{\alpha+1}}_{L^{4}}}{\|f_{n}\|^{4}_{L^{4}}}=\|f_{n}\|^{-\frac{4}{\alpha+1}}_{L^{4}},
\]
it follows that $\|f_{n}\|_{L^{4}}\gtrsim 1$.  Therefore,  we  infer  from Lieb's compactness lemma, that, after a translation if necessary, $\left\{f_{n}\right\}_{n\in \N}$  converges weakly in $H^{1}(\R^{3})$ and a.e. to a function $f\neq 0$. In particular, $\|f\|_{L^{2}}>0$, 
$\|f\|_{L^{4}}>0$, $\|f\|_{L^{6}}>0$ and $\|\nabla f\|_{L^{2}}>0$. Here we have used that 
$\|f\|^{4}_{L^{4}}\lesssim \|f\|_{L^{2}}\| \nabla f\|^{3}_{L^{2}}$. Now we put
\[\gamma=\lim_{n\rightarrow \infty}\|f_{n}\|_{L^{6}}, \quad \|f\|_{L^{2}}=m,  \quad \|\nabla f\|_{L^{2}}=t.\]
It is clear that $\gamma>0$. Indeed, $1\lesssim \|f_{n}\|^{4}_{L^{4}}\leq \|f_{n}\|^{3}_{L^{6}}$. Moreover,
$m$, $t\in (0,1]$ and 
\begin{equation}\label{Rp}
C^{-1}_{\alpha}=\gamma^{\frac{3\alpha}{\alpha+1}} \(\lim_{n\rightarrow \infty}N(f_{n})\)^{-1}.
\end{equation}
We  introduce  the  remainder $r_{n}:=f_{n}-f$. We have
\[\lim_{n\rightarrow \infty}\|r_{n}\|^{2}_{L^{2}}=1-m^{2}, \quad \lim_{n\rightarrow \infty}\|\nabla r_{n}\|^{2}_{L^{2}}=1-t^{2}.\]
The  weak convergence in $H^{1}(\R^{3})$ and Br\'ezis-Lieb lemma implies that (see \cite[p. 430]{AntonelliSparber2011})
\[\begin{split}
\lim_{n\rightarrow \infty}N(r_{n})&=\lim_{n\rightarrow \infty}[N(r_{n})-N(f_{n})]+\lim_{n\rightarrow \infty}N(f_{n})\\
&=-N(f)+\gamma^{\frac{3\alpha}{\alpha+1}}C_{\alpha}.
\end{split}\]
Thus,
\[C^{-1}_{\alpha}\leq \lim_{n\rightarrow \infty}W(r_{n})\leq 
\frac{(1-m^{2})^{\frac{1}{2}}(1-t^{2})^{\frac{3}{2}\( \frac{1}{\alpha+1}\)}\gamma^{\frac{3\alpha}{\alpha+1}}}{-N(f)+\gamma^{\frac{3\alpha}{\alpha+1}}C_{\alpha}}.\]
Now, since
\[C^{-1}_{\alpha}N(f)\leq m^{\frac{1}{2}}t^{\frac{3}{\alpha+1}}\|f\|_{L^{6}}^{\frac{3\alpha}{\alpha+1}}\leq m \,t^{\frac{3}{\alpha+1}}\gamma^{\frac{3\alpha}{\alpha+1}},\]
and $\gamma>0$ it follows that
\[(1-m^{2})^{\frac{1}{2}}(1-t^{2})^{\frac{1}{2}\( \frac{3}{\alpha+1}\)}+ m \,t^{\(\frac{3}{\alpha+1}\)}\geq 1.\]
Finally, as $\alpha>0$ and $m$, $t\in (0,1]$, it is not difficult to show that the inequality above implies that $m=1$ and $t=1$. Thus,
$\left\{f_{n}\right\}_{n\in \N}$ converges to  $f$ strongly in $H^{1}(\R^{3})$, which  means that $f$ is a minimizer.
This completes the proof of theorem.
\end{proof}

\begin{corollary}[Existence of standing waves]\label{Thecollo}
There exists at least one non-negative solution $Q\in H^{1}(\R^{3})$ of the problem \eqref{Ellip}. Moreover, 
if $\lambda_{2}<{0}$, then  $Q$ is (up to translation) axially symmetric with respect to the $x_{3}$-axis;
if $\lambda_{2}>{0}$, then  $Q$ is (up to translation) radially symmetric in the $(x_{1},x_{2})$-plane.
\end{corollary}
\begin{proof}
Let $f$ be the minimizer of the  Weinstein functional given in Theorem \ref{TheGN}. As $\|\nabla |f|\|^{2}_{L^{2}}\leq \|\nabla f\|^{2}_{L^{2}}$
we have $W(|f|)\leq W(f)$. This implies that $f(x)\geq 0$. Now, since $f$ is a minimizer, it follows that $f$ is a solution of the Euler-Lagrange equation 
\[\left.\frac{d}{d\epsilon}W(f+\epsilon\phi)\right|_{\epsilon=0}=0, \quad \mbox{for all $\phi\in C^{\infty}_{0}(\R^{3})$},\]
and so we have that $f$ satisfies the equation
\[\begin{split}
\frac{N^{\prime}(f)}{N(f)}=\|f\|^{-2}_{L^{2}}f+\frac{3\alpha}{\alpha+1}\|f\|^{-6}_{L^{6}}|f|^{4}f
+\frac{3}{\alpha+1}\|\nabla f\|^{-2}_{L^{2}}(-\Delta  f).
\end{split}\]
Since (see \cite[Lemma 3.1.]{AntonelliSparber2011})
\[N^{\prime}(f)=-4[\lambda_{1}|f|^{2}f+\lambda_{2}(K\ast|f|^{2})f]\]
we obtain
\[\begin{split}
-\Delta  f+\frac{4(\alpha+1)}{3}\frac{\lambda_{1}|f|^{2}f+\lambda_{2}(K\ast|f|^{2})f}{N(f)}+
\alpha\|\nabla f\|^{2}_{L^{2}}\|f\|^{-6}_{L^{6}}|f|^{4}f\\
+\( \frac{\alpha+1}{3}\)\|\nabla f\|^{2}_{L^{2}}\|f\|^{-2}_{L^{2}}f=0.
\end{split}\]
Next we rescale $f$ to a solution of the equation \eqref{Ellip}. Indeed, we set $Q(x)=af(bx)$, where
\[ a^{-2}=\frac{4(1+\alpha)}{3\alpha}\frac{\|f\|^{6}_{L^{6}}}{N(f)}\quad \mbox{and}\quad 
b^{-2}= \frac{16(1+\alpha)^{2}}{9\alpha}\frac{\|f\|^{2}_{L^{2}}\|f\|^{6}_{L^{6}}}{N(f)^{2}}.
\]
Then $Q$ is solution of 
\[-\Delta Q+\lambda_{1}Q^{3}
+\lambda_{1}(K\ast|Q|^{2})Q+Q^{5}+\omega Q=0, \quad \mbox{where} \quad  
\omega=\frac{3\alpha}{16(1+\alpha)}\frac{N(f)^{2}}{\|f\|^{2}_{L^{2}}\|f\|^{6}_{L^{6}}}.
\]
Finally, the symmetry assertions can be found in \cite[Proposition 3.4 (2)]{LuoSty2020}, and this completes the proof of corollary.
\end{proof}

\begin{corollary}[The sharp constant $C_{\alpha}$]\label{GNcoro}
Let $0<\alpha<\infty$. Then  the sharp constant $C_{\alpha}$ in the $\alpha$-Gagliardo-Nirenberg-H\"older inequality \eqref{SGNH}
is explicitly given by 
\begin{equation}\label{Opt}
C_{\alpha}=\frac{4(1+\alpha)}{3\alpha^{\frac{\alpha}{2(1+\alpha)}}}
\frac{\|\nabla Q_{\alpha}\|^{\frac{\alpha-1}{\alpha+1}}_{L^{2}}}{\|Q_{\alpha}\|_{L^{2}}},
\end{equation}
where $Q_{\alpha}$ is an optimizer of the problem \eqref{GNI} with 
${\|Q_{\alpha}\|^{6}_{L^{6}}}={\alpha}{\|\nabla Q_{\alpha}\|^{2}_{L^{2}}}$.
\end{corollary}
\begin{proof}
Let $f$ be an optimizer of the problem \eqref{GNI}. We put $\Gamma(f):=\frac{\|f\|^{6}_{L^{6}}}{\|\nabla f\|^{2}_{L^{2}}}$. 
By  Pohozaev identities \eqref{Pho1}-\eqref{Pho23} we have 
\begin{equation}\label{Ne0}
N(f)=\frac{4}{3}(1+\Gamma(f))\|\nabla f\|^{2}_{L^{2}}.
\end{equation}
Since ${\|f\|^{6}_{L^{6}}}=\Gamma(f){\|\nabla f\|^{2}_{L^{2}}}$,  it follows that
\[\begin{split}
C_{\alpha}=\frac{1}{W(f)}=\frac{N(f)}{\|f\|_{L^{2}}\|\nabla f\|^{\frac{3}{1+\alpha}}_{L^{2}}\|f\|_{L^{6}}^{\frac{3\alpha}{1+\alpha}}}\\
=\frac{4(1+\Gamma(f))}{3\Gamma(f)^{\frac{\alpha}{2(1+\alpha)}}}
\frac{\|\nabla f\|^{\frac{\alpha-1}{\alpha+1}}_{L^{2}}}{\|f\|_{L^{2}}}.
\end{split}\]
It is easy to show that there exist $a_{0}$, $b_{0}>0$ such that $\Gamma(a_{0}f(b_{0}\,x))=\alpha$. We set $Q_{\alpha}(x):=a_{0}f(b_{0}\,x)$. By using the fact that the Weinstein functional $W$ is invariant under the scaling $f(x)\mapsto af(b\,x)$, we see that
$Q_{\alpha}$ is an optimizer of \eqref{GNI} and 
\[\begin{split}
C_{\alpha}
=\frac{4(1+\alpha)}{3\alpha^{\frac{\alpha}{2(1+\alpha)}}}
\frac{\|\nabla Q_{\alpha}\|^{\frac{\alpha-1}{\alpha+1}}_{L^{2}}}{\|Q_{\alpha}\|_{L^{2}}},
\end{split}\]
which finishes the proof.
\end{proof}

\begin{remark}\label{Ez0}
We observe that $E(Q_{1})=0$. Indeed, since (see  \eqref{Ne0}) $N(Q_{1})=\frac{8}{3}\|\nabla Q_{1}\|^{2}_{L^{2}}$
and $\| Q_{1}\|^{6}_{L^{6}}=\|\nabla Q_{1}\|^{2}_{L^{2}}$, it follows that
\[E(Q_{1})=\frac{1}{2}\|\nabla Q_{1}\|^{2}_{L^{2}}-\frac{1}{4}N(Q_{1})+\frac{1}{6}\| Q_{1}\|^{6}_{L^{6}}=0.\]
\end{remark}

\section{Description of the region of scattering}\label{SC1}
In this section we  study some properties of the set $\K$ defined by \eqref{DefK}.
Through the rest of the paper $Q_{1}$ is the optimizer of the problem \eqref{GNI} given in Corollary \ref{Thecollo} with $\alpha=1$. Moreover,
we set $S(x):=\frac{1}{\sqrt{2}}Q_{1}(\frac{\sqrt{3}}{2}x)$.

\subsection*{Variational analysis}
We consider the infimum function  $d:[0,\infty)\rightarrow \R$,
\begin{equation}\label{VpG}
d({m}):=\inf\left\{E(u):\,\, u\in H^{1}(\R^{3}),\,\, M(u)=m  \right\},
\end{equation}
and we define the number
\[m_{\ast}:=\sup\left\{m>0: d({m})=0\right\}.\]

In the following result, we will study some properties of the variational problem \eqref{VpG}.
\begin{proposition}\label{EGE}
Let $\lambda_{1}$, $\lambda_{2}$ belong to the unstable regime \eqref{UR}.\\
{\rm (i)} If $0\leq m\leq M(Q_{1})$, then $d({m})=0$.\\
{\rm (ii)} If $m> M(Q_{1})$, then $d({m})<0$. In particular, $m_{\ast}=M(Q_{1})$. \\
{\rm (iii)} The function $d:[0,\infty)\rightarrow \R$  is  non-increasing,  non-positive, concave and thus continuous on $[0,\infty)$.
Finally, if $m\geq M(Q_{1})$, then the variational problem \eqref{VpG} is well-defined and there exists $v\in H^{1}(\R^{3})$ such that $d({m})=E(v)$. Moreover, $v$ satisfies $I(v)=0$.
\end{proposition}
\begin{proof}
Notice that $d({m})=0$ when $m=0$. Now it is easy to see that $d({m})\leq 0$ for $m>0$. Indeed, consider
the scaled functions $v_{s}(x):=s^{\frac{3}{2}}v(sx)$. Then $M(v_{s})=M(v)$ and
\[E(v_{s})=\frac{s^{2}}{2}\int_{\R^{3}}|v|+s^{3}\frac{\lambda_{1}}{4}|v|^{4}
+s^{3}\frac{\lambda_{2}}{4}(K\ast|v|^{2})|v|^{2}dx+\frac{s^{6}}{6}|v|^{6}dx\rightarrow 0
\]
as $s\rightarrow0$. This show that $d(m)\leq 0$. On the other hand, by Gagliardo-Nirenberg inequality \eqref{SGNH} with $\alpha=1$ we have
\[-\lambda_{1}\|u\|^{4}_{L^{4}}-\lambda_{2}\|(K\ast|u|^{2})|u|^{2}\|_{L^{1}}\leq 
\frac{8}{3}\( \frac{M(u)}{M(Q_{1})} \)^{\frac{1}{2}}\|\nabla u\|^{\frac{3}{2}}_{L^{2}}\|u\|^{\frac{3}{2}}_{L^{6}},  \]
and by young's inequality we find
\begin{equation}\label{IIN}
\begin{split}
E(u)\geq \frac{1}{2}\|\nabla u\|^{2}_{L^{2}}+\frac{1}{6}\|u\|^{6}_{L^{6}}-\frac{2}{3}
\( \frac{M(u)}{M(Q_{1})} \)^{\frac{1}{2}}\|\nabla u\|^{\frac{3}{2}}_{L^{2}}\|u\|^{\frac{3}{2}}_{L^{6}}\\
\geq \left[1-\( \frac{M(u)}{M(Q_{1})} \)^{\frac{1}{2}} \right]\left[\frac{1}{2}\|\nabla u\|^{2}_{L^{2}}+\frac{1}{6}\|u\|^{6}_{L^{6}}\right].
\end{split}
\end{equation}
Then for any $u\in H^{1}(\R^{3})$ with $0\leq m\leq M(Q_{1})$ we have $E(u)\geq 0$. Hence Item  {\rm (i)} is true. 
We now prove Item (ii). First, we note that $E(Q_{1})=0$ (see Remark \ref{Ez0}), which implies by \eqref{IIN} that $d(M(Q_{1}))=E(Q_{1})=0$. 
Next, assume that $m>M(Q_{1})$. We set $Q^{s}_{1}(x)=s^{-\frac{1}{2}}Q_{1}(s^{-1}x)$, where $s^{2}=m/M(Q_{1})$. It is clear that
$s>1$, $M(Q^{s}_{1})=m$ and 
\[E(Q^{s}_{1})=E(Q_{1})-\( s-1\)N(Q_{1}).\]
Since $N(Q_{1})>0$ and $E(Q_{1})=0$, we infer that $d(m)<0$ when $m>M(Q_{1})$. In particular, $d(m)$ is non-positive for all $m\geq 0$.
Finally, let $m_{1}<m_{2}$. Since $\lambda_{1}$, $\lambda_{2}$ belong to the unstable regime \eqref{UR}, we infer that there exists $u\in H^{1}(\R^{3})$ such that $M(u)=m_{1}$ and $N(u)>0$ (see Lemma \ref{Lne}). We define $u^{s}(x):=s^{-\frac{1}{2}}u(s^{-1}x)$ with $s^{2}:=m_{1}/m_{2}$. Therefore, we have
$M(u^{s})=m_{2}$ and
\[E(u^{s})=E(u)-\( s-1\)N(u)<E(u).\]
This implies that $d(m_{1})\leq d(m_{2})$.  One can find a proof of continuity and concavity of the 
function $d$ in \cite[Lemmas 8.2 and 8.3]{LuoSty2020}. Finally, Theorem 3.3 (2) in \cite{LuoSty2020}  establishes the existence of at least one  minimizer
for the variational problem \eqref{VpG}. Indeed, in Theorem 3.3 (2) of that paper it is shown that \eqref{VpG} has at least one  minimizer
when $m>m_{\ast}$. Thus, by Item (ii) we infer that the infimum $d(m)$ is achieved for all $m\geq M(Q_{1})$. Moreover, 
by Proposition 4.1 in \cite{LuoSty2020}
we see that $I(v)=0$, which completes the proof of proposition.
\end{proof}

We recall the following minimization problem defined in the introduction:
\begin{equation}\label{VVp}
\EE(m):=\inf\left\{E(u): u\in H^{1}(\R^{3}),\, M(u)=m \,\, \mbox{and}\,\, I(u)=0 \right\}. 
\end{equation}
By definition, $\EE(m)=\infty$ when the set $\left\{M(u)=m \,\, \mbox{and}\,\, I(u)=0 \right\}$ is empty.
In the following result, we prove some properties of function $\EE(m)$.  We recall that $S(x):=\frac{1}{\sqrt{2}}Q_{1}(\frac{\sqrt{3}}{2}x)$.
\begin{theorem}\label{SK}
Let $(M(u), E(u))\in \K$. Then the following statements are true.\\
(i) $I(u)>0$, where the functional $I$ is defined in \eqref{Virial}.\\
(ii) If $0<m<M(S)$, then  $\EE(m)=\infty$.\\
(iii) If $M(S)\leq m<M(Q_{1})$, then  $0<\EE(m)<\infty$. \\
(iv)  If $m\geq M(Q_{1})$ then $\EE(m)=d(m)$. In particular, $\EE(M(Q_{1}))=0$.\\
Moreover,  the function $\EE$ is  strictly decreasing on the interval $[M(S), M(Q_{1})]$. 
\end{theorem}

\begin{remark}\label{ExisGS}
By applying a similar argument as in \cite{LuoSty2020, BellazziniJeanjean2016}, one should prove that the minimization problem
\eqref{VVp} is achieved when $M(S)\leq m< M(Q_{1})$. 
Observe that the existence of minizers for \eqref{VVp} plays no role in the proof of the scattering result of Theorem \ref{TheS}.
\end{remark}

In order to prove the Theorem \ref{SK} we need the following lemmas.

\begin{lemma}\label{LG1}
Let $u\in H^{1}(\R^{3})\setminus\left\{0\right\}$. Assume that either \\
(i) $I(u)<0$ or \\
(ii) $I(u)=0$ and $\Gamma(u)<\frac{1}{3}$, where $\Gamma(u)=\frac{\|u\|^{6}_{L^{6}}}{\|\nabla u\|^{2}_{L^{2}}}$.\\
Then there exists $s>1$ such that $I(u^{s})=0$,  $\Gamma(u^{s})\geq\frac{1}{3}$ and 
$E(u^{s})<E(u)$, where $u^{s}(x):=s^{\frac{3}{2}}u(sx)$. 
\end{lemma}
\begin{proof}
A simple calculation shows that
\begin{equation}\label{IE}
\begin{split}
I(u^{s})=s\frac{d}{ds}E(u^{s})=s^{2}\|\nabla u\|^{2}_{L^{2}}+
s^{3}\frac{3}{4}\lambda_{1}\|u\|^{4}_{L^{4}}\\
+s^{3}\frac{3}{4}\lambda_{2}\|(K\ast|u|^{2})|u|^{2}\|_{L^{1}}+
s^{6}\|u\|^{6}_{L^{6}}.
\end{split}
\end{equation}
It follows from \eqref{IE} that $I(u^{s})\rightarrow \infty$ as $s$ goes to $\infty$. If $I(u)<0$, 
then by continuity there exists $s_{0}>0$ such that   $I(u^{s_{0}})=0$ and $I(u^{s_{0}})<0$ 
for all $s\in [1, s_{0})$. Thus, by \eqref{IE} we have $E(u^{s_{0}})<E(u)$.
Now, since $\left.\frac{d}{ds}I(u^{s})\right|_{s=s_{0}}\geq 0$ and  $I(u^{s_{0}})=0$ we obtain 
\[
\begin{split}
0\leq& 2s_{0}\|\nabla u\|^{2}_{L^{2}}+s_{0}^{2}\frac{9}{4}\lambda_{1}\|u\|^{4}_{L^{4}}
+s_{0}^{2}\frac{9}{4}\lambda_{2}\|(K\ast|u|^{2})|u|^{2}\|_{L^{1}}+
6s_{0}^{5}\|u\|^{6}_{L^{6}}\\
=&-s_{0}\|\nabla u\|^{2}_{L^{2}}+3s_{0}^{5}\|u\|^{6}_{L^{6}}.
\end{split}
\]
This implies that $\Gamma(u^{s_{0}})=s^{4}_{0}\Gamma(u)\geq \frac{1}{3}$. Thus, we complete the proof of (i).
Now we assume that $I(u)=0$ and $\Gamma(u)<\frac{1}{3}$. Then, as $I(u)=0$, it follow that 
\[-\frac{3}{4}\lambda_{1}\|u\|^{4}_{L^{4}}- \frac{3}{4}\lambda_{2}\|(K\ast|u|^{2})|u|^{2}\|_{L^{1}}=(1+\Gamma(u))\|\nabla u\|^{2}_{L^{2}}.\]
Therefore, by \eqref{IE} we get 
\begin{equation*}
I(u^{s})=s^{2}(1-s)(1-\Gamma(u)s-\Gamma(u)s^{2}-\Gamma(u)s^{3})\|\nabla u\|^{2}_{L^{2}}.
\end{equation*}
Notice that $\left.\frac{d}{ds}I(u^{s})\right|_{s=1}=(3\Gamma(u)-1)\|\nabla u\|^{2}_{L^{2}}$. Since $\Gamma(u)<\frac{1}{3}$, we have $\left.\frac{d}{ds}I(u^{s})\right|_{s=1}<0$. It follows that $I(u^{s})<0$ and $E(u^{s})<E(u)$
for sufficient small $s>1$. Therefore, by (i) we infer that there exists $s_{1}>1$ such that $I(u^{s_{1}})=0$,  $\Gamma(u^{s_{1}})\geq\frac{1}{3}$ and $E(u^{s_{1}})<E(u)$.
\end{proof}

\begin{lemma}\label{L53}
Let $m>0$. If $u\in H^{1}(\R^{3})$  satisfies $0<M(u)<m$ and $I(u)=0$, then there exists $v\in H^{1}(\R^{3})$ such that
\begin{equation}\label{Ax1}
\begin{split}
M(v)=m,\quad E(v)\leq E(u)-\( \frac{m-M(u)}{6M(u)}\)\|\nabla u\|^{2}_{L^{2}}\quad \mbox{and}\quad I(v)=0.
\end{split}
\end{equation}
\end{lemma}
\begin{proof}
We adapt here the proof given in \cite[Lemma 5.4]{KillipOhPoVi2017}. By Lemma \ref{LG1}, we may assume that $\Gamma(u)\geq \frac{1}{3}$.
Consider the function 
\[ f(\tau)=\frac{(1+\tau^{2}\Gamma(u))^{2}}{\tau (1+\Gamma(u))^{2}},\quad \tau>1.
\]
As $ f^{\prime}(\tau)>0$,  there exist a unique $\tau_{0}>1$ such that $m=f(\tau_{0})M(u)$. Also, since $f^{\prime}(\tau)\leq 3\Gamma(u)\tau^{2}-1$ for all $\tau>1$, it is not difficult to show that
\begin{equation}\label{Axv2}
\frac{m-M(u)}{M(u)}\leq \Gamma(u)(\tau_{0}^{3}-1)-(\tau_{0}-1).
\end{equation}
Now we set $v(x):=\sqrt{\tau_{0}\sigma}u(\sigma x)$ with $\sigma=\sqrt{\tau_{0}[f(\tau_{0})]^{-1}}$. It follows that
\begin{equation}\label{Fcs}
\begin{split}
\|\nabla v\|^{2}_{L^{2}}&=\tau_{0}\|\nabla u\|^{2}_{L^{2}}, \quad \| v\|^{4}_{L^{4}}=\tau^{2}_{0}\sigma^{-1}\| u\|^{4}_{L^{4}},\quad 
\| v\|^{2}_{L^{2}}=\tau_{0}\sigma^{-2}\| u\|^{2}_{L^{2}},
\\
\| v\|^{6}_{L^{6}}&=\tau^{3}_{0}\| u\|^{6}_{L^{6}}, \quad \|(K\ast|v|^{2})|v|^{2}\|_{L^{1}}=\tau^{2}_{0}\sigma^{-1}\|(K\ast|u|^{2})|u|^{2}\|_{L^{1}}.
\end{split}
\end{equation}
Thus, $M(v)=m$. Moreover, since $I(u)=0$, from straightforward calculations and 
\eqref{Fcs} we obtain that $I(v)=0$. Finally, by using the fact that 
\[\begin{split}
\frac{3}{4}\left\{\lambda_{1}\| u\|^{4}_{L^{4}}+\lambda_{2}\|(K\ast|u|^{2})|u|^{2}\|_{L^{1}}\right\}=-(1+\Gamma(u))\|\nabla u\|^{2}_{L^{2}} \quad \mbox{and}\\
\| u\|^{6}_{L^{6}}=\Gamma(u)\|\nabla u\|^{2}_{L^{2}},
\end{split}\]
 from \eqref{Fcs} and \eqref{Axv2} we have
\[\begin{split}
E(v)&=E(u)-\frac{1}{6}[\Gamma(u)(\tau_{0}^{3}-1)-(\tau_{0}-1)]\|\nabla u\|^{2}_{L^{2}}\\
    &\leq E(u)-\(\frac{m-M(u)}{M(u)}\)\|\nabla u\|^{2}_{L^{2}}.
\end{split}\]
This completes the proof of lemma.
\end{proof}

\begin{proof}[{Proof of Theorem \ref{SK}}]
Consider $u\in H^{1}(\R^{3})$ such that $(M(u), E(u))\in \K$.
First we prove that $I(u)>0$. Note that, by definition of $\K$, $I(u)\neq 0$.
Assume by contradiction that $I(u)<0$. By Lemma \ref{LG1}, we can see that there is $s>1$
such that  $M(u^{s})=M(u)$, $I(u^{s})=0$ and $E(u^{s})<E(u)$. Then, from the definition $\EE(m)$ we have 
$\EE(m)<E(u)$, which is impossible because $(M(u), E(u))\in \K$. This proves the part (i) of theorem.\\
Now assume that $0<m<M(S)$. We set
\begin{equation}\label{Sp1}
\varphi^{a,b}(x):=aQ_{1}(bx), \quad a>0,\quad b>0, 
\end{equation}
which implies
\begin{equation}\label{F12}
\begin{split}
\|\nabla \varphi^{a,b}\|^{2}_{L^{2}}&=a^{2}b^{-1}\|\nabla Q_{1}\|^{2}_{L^{2}}, 
\quad \| \varphi^{a,b}\|^{4}_{L^{4}}=a^{4}b^{-3}\sigma^{-1}\| Q_{1}\|^{4}_{L^{4}},\quad 
\|\varphi^{a,b}\|^{2}_{L^{2}}=a^{2}b^{-3}\| Q_{1}\|^{2}_{L^{2}},\\
\| \varphi^{a,b}\|^{6}_{L^{6}}&=a^{6}b^{-3}\| Q_{1}\|^{6}_{L^{6}}, 
\quad \|(K\ast|\varphi^{a,b}|^{2})|\varphi^{a,b}|^{2}\|_{L^{1}}=a^{4}b^{-3}\|(K\ast|Q_{1}|^{2})|Q_{1}|^{2}\|_{L^{1}}.
\end{split}
\end{equation}
Since $\Gamma(Q_{1})=1$, we obtain that 
\[\Gamma(\varphi^{a,b})=\frac{a^{4}}{b^{2}}\Gamma(Q_{1})=\frac{a^{4}}{b^{2}}.\]
Similarly, as $I(Q_{1})=0$, it follows from straightforward calculation and \eqref{F12} that
\[I(\varphi^{a,b})=\(\frac{a^{2}}{b}+ \frac{a^{6}}{b^{3}}-2\frac{a^{4}}{b^{3}}\)\|\nabla Q_{1}\|^{2}_{L^{2}}.\]
In particular, we see that
\[M(S)=\frac{4}{3\sqrt{3}}M(Q_{1}), 
\quad \Gamma(S)=\frac{1}{3},
\quad I(S)=0.
\]
On the other hand, as the Weinstein functional is invariant by scaling \eqref{Sp1}, we infer that that $S$ is an optimizer for the $(\alpha=1)$-Gagliardo-Nirenberg-H\"older inequality \eqref{SGNH}. Therefore we have
\[\begin{split}
C_{1}&=\frac{-\lambda_{1}\|S\|^{4}_{L^{4}}-\lambda_{2}\|(S\ast|S|^{2})|S|^{2}\|_{L^{1}}}{\|S\|_{L^{2}}\|\nabla S\|^{\frac{3}{1+\alpha}}_{L^{2}}\|S\|^{\frac{3\alpha}{1+\alpha}}_{L^{6}}}\\
&=\frac{4^{2}}{3^{2}} \frac{\|\nabla S\|^{2}_{L^{2}} 3^{\frac{1}{4}}}{\|S\|_{L^{2}}\|\nabla S\|^{\frac{1}{2}}_{L^{2}}\|\nabla S\|^{\frac{3}{2}}_{L^{2}} }\\
&= \frac{4^{2}}{3^{2}}\(\frac{3^{\frac{1}{4}}}{\|S\|_{L^{2}}}\). 
\end{split}\]
Here we have used the fact that $\|S\|^{6}_{L^{6}}=3^{-1}\|\nabla S\|^{2}_{L^{2}}$ and
\[\begin{split}
-\lambda_{1}\|S\|^{4}_{L^{4}}-\lambda_{2}\|(S\ast|S|^{2})|S|^{2}\|_{L^{1}}
=\frac{4}{3}\left\{\|S\|^{6}_{L^{6}}+\|\nabla S\|^{2}_{L^{2}} \right\}
=\frac{4^{2}}{3^{3}}\|\nabla S\|^{2}_{L^{2}}.
\end{split}\]
Combining the $(\alpha=1)$-Gagliardo-Nirenberg-H\"older inequality \eqref{SGNH} and the Young inequality we get
\[\begin{split}
-\lambda_{1}\|u\|^{4}_{L^{4}}-\lambda_{2}\|(u\ast|u|^{2})|u|^{2}\|_{L^{1}}&\leq
\frac{4^{2}}{3^{2}}\frac{\|u\|_{L^{2}}}{\|S\|_{L^{2}}}\(3\|u\|^{6}_{L^{6}}\)^{\frac{1}{4}}\(\|\nabla u\|^{2}_{L^{2}}\)^{\frac{3}{4}}\\
&\leq \frac{4}{3}\frac{\|u\|_{L^{2}}}{\|S\|_{L^{2}}}[\|\nabla u\|^{2}_{L^{2}}+\|u\|^{6}_{L^{6}}].
\end{split}\]
Thus, if $0<M(u)<M(S)$, then $I(u)>0$. This implies that no function obeys the constrain $\left\{u\in H^{1}(\R^{3}),\, M(u)=m \,\, \mbox{and}\,\, I(u)=0\right\}$ and therefore $\EE(m)=\infty$. This proves the part (ii) of theorem.\\
Now assume that $M(S)\leq m<M(Q_{1})$. Note that if $I(u)=0$, then 
\[\|\nabla u\|^{2}_{L^{2}}\leq -\lambda_{1}\|u\|^{4}_{L^{4}}-\lambda_{2}\|(u\ast|u|^{2})|u|^{2}\|_{L^{1}} 
\lesssim  \|\nabla u\|^{3}_{L^{2}}\| u\|_{L^{2}},  
\]
this implies that $\|\nabla u\|_{L^{2}}\| u\|_{L^{2}}\gtrsim 1$. Moreover, from \eqref{Opt}, 
we can rewrite the $(\alpha=1)$-Gagliardo-Nirenberg-H\"older inequality  as
\[-\lambda_{1}\|u\|^{4}_{L^{4}}-\lambda_{2}\|(u\ast|u|^{2})|u|^{2}\|_{L^{1}} 
\leq \frac{8}{3} \(\frac{M(u)}{M(Q_{1})}\)^{\frac{1}{2}}\|\nabla u\|^{\frac{3}{2}}_{L^{2}}\| u\|^{\frac{3}{2}}_{L^{6}}.
\]
Combining these inequalities, we get 
\[\begin{split}
E(u)&\geq \frac{1}{2}\|\nabla v\|^{2}_{L^{2}}+\frac{1}{6}\|v\|^{6}_{L^{6}}-
\frac{1}{2}\(\frac{M(u)}{M(Q_{1}}\)^{\frac{1}{2}}\|\nabla u\|^{\frac{3}{2}}_{L^{2}}\| u\|^{\frac{3}{2}}_{L^{6}}\\
&\geq \left[1-\(\frac{M(u)}{M(Q_{1})}\)^{\frac{1}{2}}\right]\left\{\frac{1}{2}\|\nabla v\|^{2}_{L^{2}}+\frac{1}{6}\|v\|^{6}_{L^{6}}\right\}\\
& \gtrsim\left[1-\(\frac{M(u)}{M(Q_{1})}\)^{\frac{1}{2}}\right][M(u)]^{-1}.
\end{split}\]
Therefore, if $M(S)\leq M(u)<M(Q_{1})$ and $I(u)=0$, then $\EE(M(u))>0$. On the other hand, 
thanks to Lemma \ref{L53} with $u=S$, we infer that $\EE(m)\leq E(S)<\infty$ when $M(S)\leq M(u)<M(Q_1)$.\\
Now we will that $\EE(m)$ is strictly decreasing on the interval $[M(S), M(Q_{1}))$. Indeed, let $m_{{2}}<m_{1}$ with $m_{{2}}$, $
m_{1}\in [M(S), M(Q_{1}))$  and  $\left\{u_{n}\right\}_{n\in \N}$ be a minimizing sequence for $\EE(m_{2})$. Thus, $M(u_{n})=m_{2}$,
$I(u_{n})=0$ and $E(u_{n})\rightarrow \EE(m_{2})$. As $I(u_{n})=0$, by applying the same argument as above we have $\|\nabla u_{n}\|^{1}_{L^{2}}\geq C/m_{2}$, where the constant $C>0$ does not depend on $n$. Therefore, by Lemma \ref{L53} we obtain a sequence $\left\{v_{n}\right\}_{n\in \N}$ such that
$M(v_{n})=m_{1}$, $I(v_{n})=0$ and 
\[E(v_{n})\leq E(u_{n})-c\frac{m_{1}-m_{2}}{6m^{2}_{2}},\]
which implies, by the definition of  $\EE(m)$ that $\EE(m_{1})<\EE(m_{2})$.  We also note that $\EE(m)=d(m)$ when  $m\geq M(Q_{1})$. Indeed, this is an immediate consequence of Proposition \ref{EGE} (iii). This implies by Proposition \ref{EGE} that the function $\EE$,
is decreasing and continuous on $[M(Q_{1}),\infty)$.
%Finally, we will show that the $\EE:[M(S),\infty)\rightarrow \R$ is lower semicontinuous on $[M(S), M(Q_{1}))$. Since $\EE$ is decreasing, 
%is equivalent to show that $\EE$ is continuous to the right. Assume by contradiction that there exists  a sequence $\left\{m_{n}\right\}_{n\in \N}\subseteq [M(S), M(Q_{1}))$ that satisfies $m_{n}\rightarrow m$ as $n\rightarrow \infty$, $m_{k}\leq m_{j}$ when $j\leq k$ and such that 
%\begin{equation}\label{ms1}
%\EE(m)>\lim_{n\rightarrow\infty}\EE(m_{n}).
%\end{equation}
 %Then there exists a sequence of minimizer $\left\{v^{\ast}_{n}\right\}$ such that
%$\EE(m_{n})=E(v^{\ast}_{n})$, $\|v^{\ast}_{n}\|^{2}_{L^{2}}=m_{n}$ and $I(v^{\ast}_{n})=0$. We set $w_{n}(x):=s_{n}v^{\ast}_{n}(x)$
% with $s^{2}_{n}=m/m_{n}$. It follows that $M(w_{n})=m$, $s_{n}\leq 1$ and $s_{n}\rightarrow 1$ as $n\rightarrow\infty$. In particular,
%$\lim_{n\rightarrow \infty}I(w_{n})=0$. Now, from \eqref{ms1} we obtain that 
%\[\lim_{n\rightarrow\infty}E(w_{n})\leq \EE(m).   \]
%Thus, $\left\{w_{n}\right\}_{n\in \N}$ is a minimizing sequence for $\EE(m)$. Then there exists $w^{\ast}$ such that 
%$E(w^{\ast})=\EE(m)$ with $\|w^{\ast}\|^{2}_{L^{2}}=m$, which is a contradiction. This completes the proof  of theorem.
\end{proof}

%%%%%%%%%%%%%%%%%%%%%%%%%%%%%%%%%%%%%%%%%%%%%%%%%%%%%%%%%%%%%%%%%%%%%%%%%%%%%%%%%%%%%%%%%%%%
\section{Perturbation theory and profile decomposition property}\label{S:Tes} 
First we introduce some useful preliminaries in the spirit of the results of 
\cite[Subsection 5.2]{KillipOhPoVi2017}.
We define the continuous map $\L: H^{1}(\R^{3})\rightarrow [0,\infty)$,
\[\L(u):=\L(M(u), E(u)):=
\begin{cases} 
E(u)+\frac{M(u)+E(u)}{\mbox{dist}\((M(u),E(u)),\Omega\)}, \quad \mbox{if} \quad (M(u),E(u))\notin \Omega\\
\infty, \quad \mbox{otherwise},
\end{cases} 
 \]
where 
\[\Omega:=\left\{(m,e)\in \R^{2}: m\geq M(S)\quad \mbox{and}\quad e\geq \EE(m)\right\}.\]
We note that $\L$ is conserved by the flow of \eqref{NLS}.
Next we collect some useful facts on $\L$.
\begin{lemma}\label{Lel}
The function $\L$ satisfies the following properties:\\
(i) $\L(u)=0$ if and only if $u\equiv 0$. Moreover, $0<\L(u)<\infty$ if and only if $(M(u),E(u))\in \K$.\\
(ii) If $0<\L(u)<\infty$, then $I(u)>0$.\\
(iii) If $M(u_{1})\leq M(u_{2})$ and $E(u_{1})\leq E(u_{2})$, then $\L(u_{1})\leq \L(u_{2})$.\\
(iv) Let $u\in H^{1}(\R^{3})$ with $\L(u)\leq \L_{0}$, where $\L_{0}\in (0, \infty)$. Then we have
\begin{equation}\label{Enl}
\|\nabla u\|^{2}_{{L}^{2}_{x}} \lesssim_{\L_{0}} E(u), \quad \mbox{and} \quad \|u\|^{2}_{{H}^{1}_{x}} \sim_{\L_{0}} 
E(u)+M(u)\sim_{\L_{0}} \L(u). 
\end{equation}
(v) If $M(u_{n})\leq M_{0}$, $E(u_{n})\leq E_{0}$, and $\L(u_{n})\rightarrow \L(M_{0}, E_{0})$, then 
$(M(u_{n}), E(u_{n}))\rightarrow (M_{0}, E_{0})$ as $n\rightarrow\infty$.
\end{lemma}
\begin{proof}
It is clear that $\L(u)=0$ if and only if $u\equiv 0$. Now notice that if $0<\L(u)<\infty$, then by definition $(M(u),E(u))\notin \Omega$. In particular, $0<M(u)<M(Q_{1})$ and $\EE(m)\leq e$. By inequality \eqref{SGNH} with $\alpha=1$ we obtain
\[-\lambda_{1}\|u\|^{4}_{L^{4}}-\lambda_{2}\|(K\ast|u|^{2})|u|^{2}\|_{L^{1}}\leq 
\frac{8}{3}\( \frac{M(u)}{M(Q_{1})} \)^{\frac{1}{2}}\|\nabla u\|^{\frac{3}{2}}_{L^{2}}\|u\|^{\frac{3}{2}}_{L^{6}}.  \]
Then, young's inequality implies that
\begin{equation}\label{Eii}
\begin{split}
E(u)\geq \frac{1}{2}\|\nabla u\|^{2}_{L^{2}}+\frac{1}{6}\|u\|^{6}_{L^{6}}-\frac{2}{3}
\( \frac{M(u)}{M(Q_{1})} \)^{\frac{1}{2}}\|\nabla u\|^{\frac{3}{2}}_{L^{2}}\|u\|^{\frac{3}{2}}_{L^{6}}\\
\geq \left[1-\( \frac{M(u)}{M(Q_{1})} \)^{\frac{1}{2}} \right]\left[\frac{1}{2}\|\nabla u\|^{2}_{L^{2}}+\frac{1}{6}\|u\|^{6}_{L^{6}}\right].
\end{split}
\end{equation}
Therefore, $e=E(u)>0$ if $0<\L(u)<\infty$, which implies (i). \\
Now assume that $0<\L(u)<\infty$. From (i) we obtain that $(M(u),E(u))\in \K$.  Thus by Theorem \ref{SK} we infer that $I(u)>0$, which proves item (ii).\\
By using the monotonicity of $\EE(m)$ (see Theorem \ref{SK}) we see that if  $M(u_{1})\leq M(u_{2})$ and $E(u_{1})\leq E(u_{2})$
then 
\begin{equation}
\label{dset}
\mbox{dist}\( (M(u_{1}), E(u_{1})), \Omega\)\geq \mbox{dist}\( (M(u_{2}), E(u_{2})), \Omega\).
\end{equation}
From the definition of $\L$, we have that $\L(u_{1})\leq \L(u_{2})$. This proves item (iii).\\
Next, we consider (iv). Suppose that $\L(u)\leq \L_{0}$. Notice that by item (i) $E(u)\geq 0$.
 First, by inequality \eqref{Eii} we see that $E(f)\geq 0$ when $M(f)=M(Q_{1})$. This implies that there exists $(m^{\ast},e^{\ast})\in \Omega$
such that $m^{\ast}=M(Q_{1})$ and $e^{\ast}=E(u)$. In particular, we have
\[  \mbox{dist}\( (M(u), E(u)), \Omega\)\leq \mbox{dist}\( (M(u), E(u)), (m^{\ast},e^{\ast})\)=M(Q_{1})-M(u).\]
Since $E(u)\geq 0$, we infer that
\begin{equation}\label{bai}
 \L(u)\geq \frac{M(u)}{M(Q_{1})-M(u)}.
\end{equation}
By inequality above, since $\frac{M(u)}{M(Q_{1})}<1$, a simple calculation shows that
\[ 
1-\sqrt{\frac{M(u)}{M(Q_{1})}}\geq \frac{1}{2\L(u)+2}.
\]
From inequality \eqref{Eii} we obtain
\begin{equation}\label{Edp}
\L(u)\geq E(u) \geq \frac{1}{2\L(u)+2}\left[\frac{1}{2}\|\nabla u\|^{2}_{L^{2}}+\frac{1}{6}\|u\|^{6}_{L^{6}}\right]\geq 
\frac{1}{4\L(u)+4}\|u\|^{2}_{\dot{H}^{1}_{x}}.
\end{equation}
Inequality above show that $\|u\|^{2}_{\dot{H}^{1}_{x}}\lesssim_{\L_{0}}E(u)$
uniformly for all $u\in H^{1}(\R^{3})$ such that $\L(u)\leq \L_{0}$. In particular, by \eqref{bai}
and \eqref{Edp} we get $\|u\|^{2}_{{H}^{1}_{x}}\lesssim_{\L_{0}} \L(u)$
Notice also that since $N(u)\leq C\|u\|^{4}_{L^{4}}$, by the Sobolev embedding we have
\[E(u)= \frac{1}{2}\|u\|^{2}_{\dot{H}^{1}_{x}}+\frac{1}{6}\|u\|^{6}_{L^{6}}-N(u)\lesssim   \|u\|^{2}_{{H}^{1}_{x}} \(1+\L(u)^{2}\).\]
On the other hand, we have that $E(u)+M(u) \sim \|u\|^{2}_{{H}^{1}_{x}}$. Indeed, notice that $N(u)\leq C\|u\|^{4}_{L^{4}}$ for some positive constant $C$. A simple calculation shows that
\[2E(u)+\frac{2}{3}\(\frac{3}{6}C\)^{2}M(u)\geq  \|\nabla u\|^{2}_{L^{2}}+
\frac{1}{3}\int_{\R^{3}}|u|^{2}\(|u|^{2}-\frac{3}{4}C\)^{2}dx.\]
Therefore,
\[\|u\|^{2}_{\dot{H}^{1}_{x}}\lesssim E(u)+M(u) \lesssim \|u\|^{2}_{{H}^{1}_{x}} \(1+\L(u)^{2}\). \]
To complete the proof of Item (iv), we need to show that $E(u)+M(u) \sim \L(u)$. First, we note that if $\frac{4 M(u)}{M(Q_{1})}\geq 1$,
then 
\[ \L(u)\leq \(\frac{4\L_{0}}{M(Q_{1})}\) M(u)+E(u).\]
Here we have used that $E(u)\geq 0$. On the other hand, if $\frac{4 M(u)}{M(Q_{1})}\leq 1$, we see that
\[ \mbox{dist}\( (M(u), E(u)), \Omega\)\geq M(S)-M(Q_{1})=\(\frac{4}{3\sqrt{3}}-\frac{1}{4}\)\geq \frac{1}{2}M(Q_{1}).
\]
This implies, by definition of $\L$, that
\[
\L(u)\leq (1+2[M(Q_{1})]^{-1})E(u)+(2[M(Q_{1})]^{-1})M(u).
\]
Moreover, combining \eqref{bai} and \eqref{Edp} we obtain
\[E(u)+M(u)\leq [1+M(Q_{1})]\L(u),\]
which completes the proof of Item (iv). Finally, Item (v) is  immediate from the definition of the function $\L$ and \eqref{dset}.
\end{proof}

Next we have a small data scattering result.
\begin{proposition}\label{Spcc}
Let $u_{0}\in H^{1}(\R^{3})$. There exists $\delta>0$ such that if $\L(u_{0})<\delta$, then there exists a unique solution to \eqref{NLS}
in $\R\times \R^{3}$ with initial condition $u(0)=u_{0}$, which satisfies
\begin{equation}\label{Gsb}
\|u\|_{L^{10}_{t,x}} \lesssim \|\nabla u_{0}\|_{L^{2}}.
\end{equation}
In particular, the solution scatters in $H^{1}(\R^{3})$, that is, there exists $u_{\pm}$ such that
\[\lim_{t\rightarrow\infty} \|u(t)-e^{it \Delta}u_{\pm}\|_{H^{1}}=0. \]
\end{proposition}
\begin{proof}Throughout the proof, all space-time norms will be taken over $\R\times \R^{3}$.
We will show that the solution map $u\mapsto \Phi (u)$,
\[  [\Phi u](t)=e^{it\Delta}u_{0}-i\int^{t}_{0}e^{i(t-s)\Delta}F(u(s))ds,\]
where 
\[F(u)=|u|^{4}u+\lambda_{1}|u|^{2}u+\lambda_{2}(K\ast|u|^{2})u,\]
is a contraction on the function space 
\[\begin{split}
X=\left\{u\in L^{\infty}_{t}H^{1}_{x}\cap L^{2}_{t}H^{1,6}_{x}\cap L_{t,x}^{10}: \|u\|_{L^{\infty}_{t}H^{1}_{x}}\leq 2\delta^{1/2},\right. \\
\left.\|u\|_{L^{2}_{t}H^{1,6}_{x}}\leq 2A\delta^{1/2}, \quad \|u\|_{L_{t,x}^{10}}\leq 2A\delta^{1/2}\right\}, 
\end{split}\]
under the metric given by
\[d(u,v)=\|u-v\|_{L^{2}_{t}L^{6}_{x}}.\]
We observe that $X$ is closed (hence complete) under this metric. Now estimating via H\"older inequality we have that
\begin{equation}\label{E12}
\begin{split}
\|\<\nabla\> |u|^{4}u  \|_{L_{t}^{\frac{10}{9}}L^{\frac{30}{7}}_{x}}\lesssim
\| u  \|^{4}_{L_{t,x}^{10}}\|\<\nabla\> u  \|_{L_{t}^{2}L^{6}_{x}},
\end{split}
\end{equation}
and 
\begin{equation}\label{E33}
\begin{split}
\|\<\nabla\> |u|^{2}u  \|_{L_{t}^{\frac{5}{3}}L^{\frac{30}{23}}_{x}}\lesssim 
\| u  \|_{L_{t}^{\infty}L_{x}^{2}}\| |u|\<\nabla\> u  \|_{L_{t}^{\frac{5}{3}}L^{\frac{15}{4}}_{x}}\\
\lesssim \| u  \|_{L_{t}^{\infty}L^{2}_{x}}\| u  \|_{L_{t,x}^{10}}\|\<\nabla\> u  \|_{L_{t}^{2}L^{6}_{x}}.
\end{split}
\end{equation}
Moreover, by using the fact that  the operator $f\mapsto K\ast f$ is continuous on $L^{p}(\R^{3})$ for each $1<p<\infty$ we obtain
\[
\begin{split}
\|\<\nabla\> [(K\ast|u|^{2})u] \|_{L^{\frac{30}{23}}_{x}}\lesssim 
\|\<\nabla\> K\ast|u|^{2}\|_{L^{\frac{15}{4}}_{x}}\|  u \|_{L^{2}_{x}}+\| K\ast|u|^{2}  \|_{L^{\frac{5}{3}}_{x}}
\| \<\nabla\> u  \|_{L^{6}_{x}}\\
\lesssim
\| u  \|_{L_{t}^{\infty}L^{2}_{x}}\| u  \|_{L_{t,x}^{10}}\|\<\nabla\> u  \|_{L_{t}^{2}L^{6}_{x}}.
\end{split}
\]
Thus,
\begin{equation}\label{E44}
\begin{split}
\|\<\nabla\> [(K\ast|u|^{2})u]  \|_{L_{t}^{\frac{5}{3}}L^{\frac{30}{23}}_{x}}\lesssim 
\| u  \|_{L_{t}^{\infty}L^{2}_{x}}\| u  \|_{L_{t,x}^{10}}\|\<\nabla\> u  \|_{L_{t}^{2}L^{6}_{x}}.
\end{split}
\end{equation}
Then we apply Strichartz estimates to estimate
\[\begin{split}
\| \Phi u  \|_{L_{t}^{\infty}L^{2}_{x}\cap L_{t}^{2}H^{1,6}_{x}}
\lesssim
\| \<\nabla\> u   \|_{L^{2}_{x}}
+\| u  \|_{L_{t}^{\infty}L^{2}_{x}}\| u  \|_{L_{t,x}^{10}}\|\<\nabla\> u  \|_{L_{t}^{2}L^{6}_{x}}
\\+
\| u  \|^{4}_{L_{t,x}^{10}}\|\<\nabla\> u  \|_{L_{t}^{2}L^{6}_{x}}.
\end{split}\]
Similarly, combining \eqref{E12}, \eqref{E33}, \eqref{E44} and the Sobolev embedding $L_{t}^{10}\dot{H}^{1,\frac{30}{13}}_{x}\hookrightarrow L_{t,x}^{10}$ we obtain
\begin{equation}\label{Gbc}
\begin{split}
\|  \Phi u  \|_{L_{t,x}^{10}}
\lesssim\| \nabla u_{0}   \|_{L^{2}_{x}}+\| u  \|_{L_{t}^{\infty}L^{2}_{x}}\| u  \|_{L_{t,x}^{10}}\|\nabla u  \|_{L_{t}^{2}L^{6}_{x}}+
\| u  \|^{4}_{L_{t,x}^{10}}\|\nabla u  \|_{L_{t}^{2}L^{6}_{x}}.
\end{split}
\end{equation}
Since $\| u  \|^{2}_{H^{1}_{x}}\lesssim \L(u(t))=\L(u_{0})<\delta$, choosing $\delta$ sufficiently small and $A$ sufficiently large,
we have that the functional $\Phi $ map $X$ back to itself. Next we show that $\Phi$ is a contraction. Indeed,
first notice that estimating via H\"older as above,
\[ \| |u|^{4}u-|v|^{4}v  \|_{L_{t}^{\frac{10}{9}}L^{\frac{30}{17}}_{x}}+\||u|^{2}u-|v|^{2}v \|_{L_{t}^{\frac{5}{3}}L^{\frac{30}{23}}_{x}}
\lesssim (\delta^{4}+\delta)\| u-v \|_{L_{t}^{2}L^{6}_{x}}.
\]
Moreover,
\[\begin{split}
\|(K\ast|u|^{2})u-(K\ast|v|^{2})v\|_{L^{\frac{30}{23}}_{x}}=
\|(K\ast|u|^{2})(u-v)-(K\ast(|u|^{2}-|v|^{2})v\|_{L^{\frac{30}{23}}_{x}}\\
\lesssim \||u|^{2}\|_{L^{\frac{5}{3}}_{x}}\|u-v\|_{L^{6}_{x}}+
\|v\|_{L^{2}_{x}}\||u|^{2}-|v|^{2}\|_{L^{\frac{15}{4}}_{x}}\\
\lesssim  
\| u  \|_{L^{2}_{x}} 
\|  u \|_{L^{10}_{x}} 
\| u-v  \|_{L^{6}_{x}}+
\| v  \|_{L^{2}_{x}}
\|  u+v \|_{L^{10}_{x}}
\| u-v  \|_{L^{6}_{x}}.
\end{split}\]
Thus,
\[\|(K\ast|u|^{2})u-(K\ast|v|^{2})v\|_{L_{t}^{\frac{5}{3}}L^{\frac{30}{23}}_{x}}
\lesssim \delta\| u-v \|_{L_{t}^{2}L^{6}_{x}},
\]
which implies that 
\[\|\Phi u-\Phi v\|_{L^{2}_{t}L^{6}_{x}}\lesssim \delta\|u-v\|_{L^{2}_{t}L^{6}_{x}}.\]
Therefore, choosing $\delta$ even smaller (if necessary), we see that $\Phi$ is a contraction on $X$.
Notice that the global space-time bound \eqref{Gsb} follows from inequality \eqref{Gbc}. 
Finally, scattering follows from another application of inequalities \eqref{E12}, \eqref{E33} and \eqref{Gbc}, as above.
\end{proof}

\begin{remark}{(Persistence of regularity)}\label{Rpr}
Suppose that $u: \R\times \R^{3}\rightarrow \C$ is a solution to \eqref{NLS} such that $\|u\|_{L^{10}_{t,x}}\leq L$,
Then 
\begin{equation}\label{S1}
\||\nabla|^{s}u\|_{S^{0}(\R)}\leq C(L, M(u))\||\nabla|^{s}u(t_{0})\|_{L^{2}_{x}}, \quad s=0,1,
\end{equation}
for any $t_{0}\in \R$. Indeed, since $\|u\|_{L^{10}_{t,x}}\leq L$, given $\epsilon>0$ (to be choose below) we can decompose $\R$ into $m$ many intervals
$I_{j}=[t_{j}, t_{j+1}]$ with $\|u\|_{L^{10}_{t,x}(I_{j}\times \R^{3})}\leq \epsilon$. In the following, we take all space-time norms over $I_{j}\times \R^{3}$.
Combining the inequalities \eqref{E12}, \eqref{E33} and \eqref{Gbc}, via Strichartz estimates we obtain
\begin{equation}\label{Dfu}
\begin{split}
\||\nabla|^{s}u\|_{S^{0}(I_{j})}\lesssim \||\nabla|^{s}u(t_{j})\|_{L^{2}_{x}}
+\| u  \|_{L_{t}^{\infty}L^{2}_{x}}\| u  \|_{L_{t,x}^{10}}\||\nabla|^{s}u  \|_{L_{t}^{2}L^{6}_{x}}
+\| u  \|^{4}_{L_{t,x}^{10}}\||\nabla|^{s} u  \|_{L_{t}^{2}L^{6}_{x}}\\
\lesssim \||\nabla|^{s}u(t_{j})\|_{L^{2}_{x}}+\epsilon (\epsilon^{3}+M(u)^{1/2})\||\nabla|^{s}u\|_{S^{0}(I_{j})},
\end{split}
\end{equation}
where $\epsilon=\epsilon(M(u))>0$ is chosen small enough that $\||\nabla|^{s}u\|_{S^{0}(I_{j})}\lesssim \||\nabla|^{s}u(t_{j})\|_{L^{2}_{x}}$.
Finally, by summing over the $m$ intervals, we get \eqref{S1}.
\end{remark}

\begin{remark}\label{Sld}
Let $u$ be a solution to \eqref{NLS} on $\R\times\R^{3}$ such that  $\|u\|_{L^{10}_{t,x}(\R\times \R^{3})}< \infty$. 
Combining Remark \ref{Rpr} and Strichartz estimates (as in the proof of Proposition \ref{Spcc}), it is not difficult to show that $u$ scatters in $H^{1}(\R^{3})$.
\end{remark}

Now we have a stability result for \eqref{NLS}, which will play an important role
in the proof of Theorem \ref{TheS}.
\begin{lemma}[Stability result]\label{stabi}
Let $I\subset \R$ be a time interval containing $t_{0}$ and  let $\tilde{u}$ defined on $I\times\R^{3}$ be a solution of the problem
\[ (i\partial_{t}+\Delta) \tilde{u}=\lambda_{1}|\tilde{u}|^{2}\tilde{u}+\lambda_{2}(K\ast|\tilde{u}|^{2})\tilde{u}+|\tilde{u}|^{4}\tilde{u}+e, \quad 
\tilde{u}(t_{0})=\tilde{u}_{0}\]
for some error $e:I\times\R^{3}\rightarrow \C$. Assume also the conditions
\begin{align}\label{SCN1}
	\| \tilde{u}  \|_{L_{t}^{\infty}H^{1}_{x}(I\times\R^{3})}&\leq A\\\label{SCN22}
	\| \tilde{u}  \|_{L_{t,x}^{10}(I\times\R^{3})}&\leq L
\end{align}
for some $A$, $L>0$. Let $u_{0}\in H^{1}(\R^{3})$ such that $\|u_{0}\|_{L^{2}}\leq M$ for some  positive constant $M$.  
There exists $\epsilon_{0}=\epsilon_{0}(\mbox{A,L,M})>0$ such that if $0<\epsilon<\epsilon_{0}$ and
\begin{align}\label{SCN33}
	\|u_{0}-\tilde{u}_{0} \|_{\dot{H}^{1}_{x}}&\leq \epsilon\\\label{SCN44}
	\|\nabla e  \|_{L_{t,x}^{\frac{10}{7}}(I\times\R^{3})}&\leq \epsilon,
\end{align}
then there exists a unique global solution $u$ to \eqref{NLS} with initial data $u_{0}$ at the time $t=t_{0}$
satisfying
\begin{equation}\label{SCN55}
\|\nabla (u-\tilde{u})\|_{S^{0}(I)}\leq C(E,L,M)\epsilon.
\end{equation}
\end{lemma}
\begin{proof}
First, we show that 
\begin{equation}\label{Cbc}
\|\nabla \tilde{u}\|_{S^{0}(I)}\leq C(A,L).
\end{equation}
 Indeed, notice that from \eqref{SCN22} we may decompose $\R$ into 
$m=m(\eta, L)$ many intervals $I_{j}=[t_{j}, t_{j+1}]$ such that on each space-time slab $I_{j}\times \R^{3}$,
\[\|u\|_{L^{10}_{t,x}(I_{j}\times \R^{3})}\leq \eta  \]
for $\eta>0$ to be choosen later. By Strichartz we have (see \eqref{Dfu})
\[\begin{split}
\|\nabla \tilde{u}\|_{S^{0}(I_{j})}\lesssim \|\tilde{u}\|_{L_{t}^{\infty}H^{1}_{x}(I\times\R^{3})}
+\| \tilde{u}  \|_{L_{t}^{\infty}L^{2}_{x}}\| \tilde{u}  \|_{L_{t,x}^{10}}\|\nabla\tilde{ u}  \|_{L_{t}^{2}L^{6}_{x}}\\
+\| u  \|^{4}_{L_{t,x}^{10}}\|\nabla \tilde{u } \|_{L_{t}^{2}L^{6}_{x}}+\|\nabla e   \|_{L^{\frac{10}{7}}_{t,x}}\\
\lesssim A+\epsilon+\eta\|\nabla \tilde{u}\|_{S^{0}(I_{j})}.
\end{split}\]
Choosing $\eta$ and $\epsilon_{1}$ small enough, a standard continuity argument show that
\[ \|\nabla \tilde{u}\|_{S^{0}(I_{j})}\leq C(A,L).\]
Summation over $I_{j}$ yields the space time bound \eqref{Cbc}.
By symmetry we may assume that $t_{0}=0$. Using the estimates \eqref{SCN22}, \eqref{Cbc}, and a standard combinatorial argument, we may prove the lemma under the additional hypothesis 
\begin{equation}\label{adt}
\| \tilde{u } \|_{L_{t,x}^{10}}+\|  \nabla \tilde{u} \|_{L_{t}^{2}L^{6}_{x}}\leq \eta,
\end{equation}
for a small constant $\eta>0$ to be chosen in a moment. Let $w:=u-\tilde{u}$. It is clear that $w$ solve the initial value problem
\begin{equation}\label{PNLS}
\begin{cases} 
(i\partial_{t}+\Delta) w=F(\tilde{u}+w)-F(\tilde{u})-e,\\
w(0)=u_{0}-\tilde{u}_{0},
\end{cases} 
\end{equation}
where  
\[F(u)=|u|^{4}u+\lambda_{1}|u|^{2}u+\lambda_{2}(K\ast|u|^{2})u.\]
In the following, we take all space-time norms over $(I\cap [-t,t])\times \R^{3}$. From H\"older inequality we get
\begin{equation}\label{PI11}
\begin{split}
\|\nabla (|w+\tilde{u}|^{4}(w+\tilde{u})-|\tilde{u}|^{4}\tilde{u})\|_{L_{t}^{\frac{10}{9}}L^{\frac{30}{17}}_{x}}
\lesssim \|  \nabla w \|_{L_{t}^{2}L^{6}_{x}}[  \| \tilde{u } \|^{4}_{L^{10}_{t,x}}+\| w \|^{4}_{L^{10}_{t,x}}]\\
+\|\nabla \tilde{u} \|_{L_{t}^{2}L^{6}_{x}}[  \| \tilde{u } \|^{3}_{L^{10}_{t,x}}\| w \|_{L^{10}_{t,x}}+\| w \|^{4}_{L^{10}_{t,x}}]
\end{split}
\end{equation}
and
\begin{equation}\label{PI22}
\begin{split}
\|\nabla (|w+\tilde{u}|^{2}(w+\tilde{u})-|\tilde{u}|^{2}\tilde{u})\|_{L_{t}^{\frac{5}{3}}L^{\frac{30}{23}}_{x}}\lesssim
\|\nabla w \|_{L_{t}^{2}L^{6}_{x}}[\| \tilde{u } \|_{L^{10}_{t,x}}\| \tilde{u }  \|_{L_{t}^{\infty}L^{2}_{x}}+ 
\| w\|_{L^{10}_{t,x}}\|w \|_{L_{t}^{\infty}L^{2}_{x}}]\\
+\|\nabla \tilde{u} \|_{L_{t}^{2}L^{6}_{x}}\| w\|_{L^{10}_{t,x}} [ \| \tilde{u }  \|_{L_{t}^{\infty}L^{2}_{x}}
+ \| w\|_{L_{t}^{\infty}L^{2}_{x}}].
\end{split}
\end{equation}
On the other hand, 
\begin{equation}\label{PI33}
\begin{split}
\nabla [K\ast|w+\tilde{u}|^{2}({w+\tilde{u}})-K\ast|\tilde{u}|^{2}\tilde{u}]=
[K\ast|w+\tilde{u}|^{2}\nabla(w+\tilde{u})-K\ast|\tilde{u}|^{2}\nabla \tilde{u}]\\+
[(K\ast\nabla|w+\tilde{u}|^{2})({w+\tilde{u}})-(K\ast\nabla|\tilde{u}|^{2})\tilde{u}].
\end{split}
\end{equation}
Since  the operator $f\mapsto K\ast f$ is continuous on $L^{p}(\R^{3})$ (see Lemma \ref{LLk}), via H\"older inequality we obtain
\begin{equation}\label{PI44}
\begin{split}
\|K\ast|w+\tilde{u}|^{2}\nabla(w+\tilde{u})-K\ast|\tilde{u}|^{2}\nabla \tilde{u}\|_{L_{t}^{\frac{5}{3}}L^{\frac{30}{23}}_{x}}
\\\lesssim \|\nabla \tilde{u} \|_{L_{t}^{2}L^{6}_{x}}\| w \|_{L^{10}_{t,x}}[\| \tilde{u }  \|_{L_{t}^{\infty}L^{2}_{x}}
+ \| w\|_{L_{t}^{\infty}L^{2}_{x}}]
\\+
 \|\nabla w \|_{L_{t}^{2}L^{6}_{x}}[(\| \tilde{u }  \|_{L_{t}^{\infty}L^{2}_{x}}
+ \| w\|_{L_{t}^{\infty}L^{2}_{x}})(\| w \|_{L^{10}_{t,x}}+\| \tilde{u}\|_{L^{10}_{t,x}})],
\end{split}
\end{equation}
and
\begin{equation}\label{PI55}
\begin{split}
\|(K\ast\nabla|w+\tilde{u}|^{2})({w+\tilde{u}})-(K\ast\nabla|\tilde{u}|^{2})\tilde{u}\|_{L_{t}^{\frac{5}{3}}L^{\frac{30}{23}}_{x}}
\\\lesssim (\| w\|_{L_{t}^{\infty}L^{2}_{x}}+1)
[\|\nabla w \|_{L_{t}^{2}L^{6}_{x}}+ \|\nabla \tilde{u} \|_{L_{t}^{2}L^{6}_{x}}]
[\| w \|_{L^{10}_{t,x}}+\| \tilde{u}\|_{L^{10}_{t,x}}].
\end{split}
\end{equation}
Since $\| \tilde{u }  \|_{L_{t}^{\infty}L^{2}_{x}}
+ \| w\|_{L_{t}^{\infty}L^{2}_{x}}\lesssim_{A,M} 1$, we deduce from \eqref{PI11}, \eqref{PI22}, \eqref{PI33}, \eqref{PI44} and \eqref{PI55} that
\begin{equation}\label{UBa}
\varphi(t)\lesssim \epsilon+\eta^{2}+\eta\varphi(t)+\eta^{4}\varphi(t)+\eta\varphi(t)^{4}+\varphi(t)^{2}+\varphi(t)^{5},
\end{equation}
where we have set
\[\varphi(t):=\|w\|_{\dot{S}^{1}(I\cap[-t,t])}.\]
Thus, by \eqref{UBa} and the usual bootstrap argument yields \eqref{SCN55}, provided $\epsilon_{1}$ and $\eta$ are chosen sufficiently small depending on $E$ $L$, $M$. This completes the proof of lemma.
\end{proof}

We need the following profile decomposition property, which is a key ingredient in proving the
existence of minimal blowup solutions. 
\begin{theorem}[Linear profile decomposition]\label{Profi}
Let $\left\{f_{n}\right\}_{n\in \N}$ be a bounded sequence of $H^{1}(\R^{3})$. Up to subsequence, there exist
 $J^{\ast}\in \left\{0,1,2,\ldots\right\}\cup\left\{\infty\right\}$, non-zero profiles 
$\left\{\phi^{j}\right\}^{J^{\ast}}_{j=1}\subset \dot{H}_{x}^{1}(\R^{3})\setminus\left\{0\right\}$ and parameters
$\left\{(\lambda^{j}_{n}, t^{j}_{n}, x^{j}_{n})\right\}^{J^{\ast}}_{j=1}\subset (0,1]\times \R\times\R^{3}$ satisfying for any fixed $j$,
\begin{itemize}
	\item $\lambda^{j}_{n}\equiv 1$ or $\lambda^{j}_{n}\rightarrow 0$ and $t^{j}_{n}\equiv 0$ or $t^{j}_{n}\rightarrow\pm\infty$,
	\item If $\lambda^{j}_{n}\equiv 1$ then $\left\{\phi^{j}\right\}^{J^{\ast}}_{j=1}\subset {H}_{x}^{1}(\R^{3})$.
\end{itemize}
In addition, we can write
\begin{equation}\label{Dcom}
f_{n}=\sum^{J}_{j=1}\phi_{n}^{j}+W^{J}_{n}
\end{equation}
for each finite $1\leq J\leq J^{\ast}$, where
\begin{equation}\label{fucti}
\phi_{n}^{j}(x):=
\begin{cases} 
[e^{it^{j}_{n}\Delta}\phi^{j}](x-x^{j}_{n}), \quad \mbox{if $\lambda^{j}_{n}\equiv 1$},\\
(\lambda^{j}_{n})^{-\frac{1}{2}}[e^{it^{j}_{n}\Delta} P_{\geq (\lambda^{j}_{n})^{\theta}}\phi^{j}]\( \frac{x-x^{j}_{n}}{\lambda^{j}_{n}}\),
\quad \mbox{if $\lambda^{j}_{n}\rightarrow 0$},
\end{cases}
\end{equation}
for some $0<\theta<1$. Furthermore, the following properties hold:
\begin{itemize}
	\item Smallness of the reminder: 
	\begin{equation}\label{Sr}
\mbox{for every $\epsilon>0$ there is $J=J(\epsilon)$ such that}\,\, 
\limsup_{n\rightarrow\infty}\|e^{it\Delta}W^{J}_{n}\|_{L^{10}_{t,x}}<\epsilon.
   \end{equation}
		\item Weak convergence property:
		\begin{equation}\label{Wcp}
e^{-it^{j}_{n}\Delta}[(\lambda^{j}_{n})^{\frac{1}{2}}W^{J}_{n}(\lambda^{j}_{n}x+x^{j}_{n})]\rightharpoonup 0\quad \mbox{in}\,\,
\dot{H}^{1}_{x}, \quad \mbox{as $n\rightarrow\infty$.}
     \end{equation}
			\item The time and space sequence have a pairwise divergence property: for all $1\leq j\neq k\leq J^{\ast}$
	\begin{equation}\label{Pow}
\lim_{n\rightarrow \infty}\left[ \frac{\lambda^{j}_{n}}{\lambda^{k}_{n}}+\frac{\lambda^{k}_{n}}{\lambda^{j}_{n}} 
+\frac{|x^{j}_{n}-x^{k}_{n}|}{\lambda^{j}_{n}\lambda^{k}_{n}}+
\frac{|t^{j}_{n}(\lambda^{j}_{n})^{2}-t^{k}_{n}(\lambda^{k}_{n})^{2}|}{\lambda^{j}_{n}\lambda^{k}_{n}}\right]=\infty.
\end{equation}		
		\item Orthogonality in norms: for any $J\in \N$
		\begin{align}\label{PE11}
			\| f_{n}  \|^{2}_{L^{2}_{x}}&=\sum^{J}_{j=1}\| \phi^{j}_{n}  \|^{2}_{L^{2}_{x}}+\| W^{j}_{n}  \|^{2}_{L^{2}_{x}}+o_{n}(1),\\\label{PE22}
			\| f_{n}  \|^{4}_{L^{4}_{x}}&=\sum^{J}_{j=1}\| \phi^{j}_{n}  \|^{2}_{L^{4}_{x}}+\| W^{j}_{n}  \|^{4}_{L^{4}_{x}}+o_{n}(1),\\\label{PE33}
			\| f_{n}  \|^{6}_{L^{6}_{x}}&=\sum^{J}_{j=1}\| \phi^{j}_{n}  \|^{2}_{L^{6}_{x}}+\| W^{j}_{n}  \|^{6}_{L^{6}_{x}}+o_{n}(1),\\\label{PE44}
			\| f_{n}  \|^{2}_{\dot{H}^{1}_{x}}&=\sum^{J}_{j=1}\| \phi^{j}_{n}  \|^{2}_{\dot{H}^{1}_{x}}+
			\| W^{j}_{n}  \|^{2}_{\dot{H}^{1}_{x}}+o_{n}(1).
		\end{align}
	\end{itemize}
\end{theorem}
\begin{proof}
See Theorem 7.5 in \cite{KillipOhPoVi2017}.
\end{proof}
\begin{lemma}\label{Lde}
Under the hypotheses of Theorem \ref{Profi}, we have for any $J\in \N$,
\begin{equation}\label{PEK}
\int_{\R^{3}}(K\ast|f_{n}|^{2})|f_{n}|^{2}dx=\sum^{J}_{j=1}\int_{\R^{3}}(K\ast|\phi^{j}_{n} |^{2})|\phi^{j}_{n} |^{2}dx
+\int_{\R^{3}}(K\ast| W^{j}_{n}|^{2})| W^{j}_{n}|^{2}dx+o_{n}(1),
\end{equation}
and in particular,
\begin{equation}\label{PE}
E(f_{n})=\sum^{J}_{j=1}E(\phi^{j}_{n} )+E(W^{j}_{n})+o_{n}(1),
\end{equation}
where $o_{n}(1)\rightarrow0$ as $n\rightarrow \infty$.
\end{lemma}
\begin{proof}
The proof is essentially the same as in Proposition 4.3 and Corollary 4.4 in \cite{BellazziniForcella2019}.
\end{proof}

\section{Existence of a critical solution}\label{Sps}
In the next sections, we give the proof of Theorem \ref{TheS}. As mentioned in the introduction, our arguments are parallel those of 
\cite{KillipOhPoVi2017}. For every $\tau>0$ we define the number $B(\tau)$ as follows:
\[B(\tau):=\sup\left\{\|  u \|_{L^{10}_{t,x}(\R\times\R^{3})}:\mbox{$u$ solves \eqref{NLS} and $\L(u)\leq \tau$}\right\}.
\]
Notice that by Proposition \ref{Spcc} and \eqref{Enl}, there exists $\tau>0$ small enough such 
that if $\L(u(t))=\L(u_{0})<\tau$, then $\|u\|_{L^{10}_{t,x}(\R\times\R^{3})} \lesssim \|\nabla u_{0}\|_{L^{2}}\lesssim E(u_{0})^{\frac{1}{2}}$;
that is, taking $\tau>0$ sufficiently small implies that $B(\tau)<\infty$. 
Suppose by contraction that Theorem \ref{TheS} fails. By the monotonicity of $B$ and Remark \ref{Sld}, there exists a critical level $0<\tau_{c}<\infty$
such that
\begin{equation}\label{tcri}
\tau_{c}=\sup\left\{\tau: B(\tau)<\infty\right\}=\inf\left\{\tau: B(\tau)=\infty\right\}.
\end{equation}
From Lemma \ref{stabi}, it is not difficult to show that $B(\tau_{c})=\infty$. Since $B(\tau_{c})=\infty$, by definition of $B$, we can find a sequence of solutions of \eqref{NLS}
with $\L(u_{n})\rightarrow \tau_{c}$ such that $\|  u_{n} \|_{L^{10}_{t,x}(\R\times\R^{3})}\rightarrow \infty$ as $n\rightarrow\infty$.
By translating $u_{n}$ in the time, we can assume that 
$\|  u_{n} \|_{L^{10}_{t,x}((0,\infty)\times\R^{3})}=\|  u_{n} \|_{L^{10}_{t,x}((-\infty,0)\times\R^{3})}\rightarrow \infty$.
Now our goal is to prove the existence of critical element $u_{c}\in H^{1}(\R^{3})$, 
which is a solution of \eqref{NLS} with $u_{c}(0)=u_{c,0}$ such that $\L(u_{c})=\tau_{c}$ and 
\begin{equation}\label{blow}
\|  u_{c} \|_{L^{10}_{t,x}([0,\infty)\times\R^{3})}=\|  u_{c} \|_{L^{10}_{t,x}((-\infty,0]\times\R^{3})}=\infty.
\end{equation}
Also, we prove that the orbit $\left\{u(t):t\in \R\right\}$ is precompact in $H^{1}(\R^{3})$ modulo translations.
More specifically, we have the following result.
\begin{theorem}[Existence and compactness of a critical solution]\label{ECEC}
Let $\tau_{c}$ be defined by \eqref{tcri}. There exists $u_{c,0}\in H^{1}(\R^{3})$ with $\L(u_{c,0})=\tau_{c}$
such that if $u_{c}$ is the global solution to \eqref{NLS} with initial data $u_{c}(0)=u_{c,0}$, then \eqref{blow}
holds. Moreover, the orbit $\left\{u(t):t\in \R\right\}$ is relatively compact in $H^{1}(\R^{3})$ modulo translations.
\end{theorem}

The main step in the proof of the Theorem \ref{ECEC} is the following result, related with the linear profile decomposition
Theorem \ref{Profi}. 

\begin{proposition}\label{PSmc}
Let $\left\{u_{n}\right\}_{n\in \N}\subset H^{1}(\R^{3})$ be a sequence of global solutions of \eqref{NLS} and suppose that
$\lim_{n\rightarrow\infty}\L(u_{n})=\tau_{c}$ and
\begin{equation}\label{Blowp2}
\lim_{n\rightarrow\infty}\|  u_{n} \|_{L^{10}_{t,x}([t_{n},\infty)\times\R^{3})}=\|  u_{n} \|_{L^{10}_{t,x}((-\infty,t_{n}]\times\R^{3})}=\infty,
\end{equation}
where $\left\{t_{n}\right\}_{n\in \N}\subset \R$. It follows that there exists a subsequence, which we still denote by $\left\{u_{n}\right\}_{n\in \N}$,  and
$\left\{x_{n}\right\}_{n\in \N}\subset \R^{3}$ such that $\left\{u_{n}(t_{n},\cdot+x_{n})\right\}$
converges in  $H^{1}(\R^{3})$.
\end{proposition}

We assume the Proposition \ref{PSmc} for the moment, and proceed with the proof of Theorem \ref{ECEC}

\begin{proof}[Proof of Theorem \ref{ECEC}]
As $B(\tau_{c})=\infty$, by definition it means that there exists a sequence  of initial data $u_{n}(0)$ such that
$\lim_{n\rightarrow\infty}\L(u_{n}(0))=\tau_{c}$. Moreover, if $u_{n}(t)$ is the corresponding solution to \eqref{NLS} with data initial 
$u_{n}(0)$, then we can have that \eqref{Blowp2} holds. By translating $u_{n}$ in the time, we may assume that
\begin{equation}\label{L10}
\|  u_{n} \|_{L^{10}_{t,x}([0,\infty)\times\R^{3})}=\|  u_{n} \|_{L^{10}_{t,x}((-\infty,0]\times\R^{3})}\rightarrow \infty, \quad \text{as $n\rightarrow\infty$}.
\end{equation}
Using Proposition \ref{PSmc}, it follows that the sequence $u_{n}(0)$ has a converging subsequence in $H^{1}(\R^{3})$ modulo
spatial translations. Thus, by symmetry, we may assume that there exists $u_{0,c}\in H^{1}(\R^{3})$ such that $u_{n}(0)\rightarrow u_{0,c}$ in $H^{1}(\R^{3})$ as $n\rightarrow\infty$. Let $u_{c}$ be the global solution to \eqref{NLS} with initial data $u_{0,c}$, then 
$\L(u_{c}(t))=\L(u_{0,c})=\tau_{c}$. In particular, combining \eqref{Enl}, \eqref{S1} and Lemma \ref{stabi}, we obtain that 
\begin{equation}\label{Coc}
\|  u_{c} \|_{L^{10}_{t,x}([0,\infty)\times\R^{3})}=\|  u_{c} \|_{L^{10}_{t,x}((-\infty,0]\times\R^{3})}=\infty.
\end{equation}
Otherwise, $\|  u_{n} \|_{L^{10}_{t,x}(\R\times\R^{3})}\lesssim 1$ uniformly in $n$, which contradicts with \eqref{L10}.

Finally, we consider the precompactness of $\left\{u(t):t\in \R\right\}$ in $H^{1}(\R^{3})$ modulo translations. Indeed, since $u$
is locally in $L^{10}_{t,x}$ (see Proposition \ref{LWP}) we infer that
\[\|  u_{n} \|_{L^{10}_{t,x}([T_{n},\infty)\times\R^{3})}=\|  u_{n} \|_{L^{10}_{t,x}((-\infty,T_{n}]\times\R^{3})}=\infty,
\]
for some time sequence $\left\{T_{n}\right\}_{n\in \N}\subset \R$. Hence, from Proposition \ref{PSmc} we see that  $u_{c}(T_{n})$ 
is precompact in  $H^{1}(\R^{3})$ modulo translations.
\end{proof}

The remainder of this section is dedicated to the proof of Proposition \ref{PSmc}.
In the proof of Proposition \ref{PSmc}, we use the following lemma.
\begin{lemma}\label{Lpsci}
Let $\left\{\lambda_{n}\right\}_{n\in \N}\subset (0,\infty)$ satisfy $\lambda_{n}\rightarrow 0$ as $n\rightarrow\infty$.
Let $\phi\in H^{1}(\R^{3})$, $\left\{x_{n}\right\}_{n\in \N}\subset \R^{3}$ and the time sequence $\left\{t_{n}\right\}_{n\in \N}\subset \R$
such that either $t_{n}\equiv0$ or $t_{n}\rightarrow \pm\infty$. For $0<\theta<1$, we define
\[\phi_{n}(x)=\lambda^{-\frac{1}{2}}_{n}[e^{it_{n}\Delta}P_{\geq \lambda^{\theta}_{n}}\phi]\Big(  \frac{x-x_{n}}{\lambda_{n}}\Big).\]
Then taking $n$ large enough, we have that the corresponding global solution $u(t)$  of \eqref{NLS} with initial data
$u_{0}=\phi_{n}$ satisfies
\begin{equation}\label{Csi}
\|\nabla u_{n}   \|_{S^{0}(\R)}\leq C(\|\phi\|_{\dot{H}^{1}_{x}}).
\end{equation}
In addition, for every $\epsilon>0$ there exist a number $N=N(\epsilon)\in \N$ and smooth compact supported functions 
$\chi_{\epsilon}$, $\psi_{\epsilon}$ and $\zeta_{\epsilon}$ on $\R\times\R^{3}$ satisfying for all $n\geq N$
\begin{align}\label{Small11}
\| u_{n}(t,x)-\lambda^{-1/2}_{n}\chi_{\epsilon}\Big(\frac{t}{\lambda^{2}_{n}}+t_{n}, \frac{x-x_{n}}{\lambda_{n}}\Big)  \|_{L^{10}_{t,x}}&<\epsilon,\\
	\label{Small22}
\| \nabla u_{n}(t,x)-\lambda^{-3/2}_{n}\psi_{\epsilon}\Big(\frac{t}{\lambda^{2}_{n}}+t_{n}, \frac{x-x_{n}}{\lambda_{n}}\Big)  \|_{L^{\frac{10}{3}}_{t,x}}&<\epsilon.
\end{align}
\end{lemma}
\begin{proof}
By symmetry, we may assume that $x_{n}\equiv 0$. Following the same argument as in \cite[Proposition 8.3]{KillipOhPoVi2017} 
we get that for every $n\in \N$ there exists a unique global solution $\tilde{u}_{n}$ of 
the defocusing quintic Schr\"odinger equation
\[(i\partial_{t}+\Delta) v=|v|^{4}v,\]
such that
\begin{equation}\label{Spli}
\begin{split}
&\|\nabla \tilde{u}_{n}   \|_{S^{0}(\R)}\leq C(\|\phi\|_{\dot{H}^{1}_{x}})\quad \mbox{and}\quad
\|\tilde{u}_{n}\|_{S^{0}(\R)}\leq C(\|\phi\|_{\dot{H}^{1}_{x}})\lambda^{1-\theta}_{n},\\
&\| \tilde{u}_{n}(0)-\phi_{n} \|_{\dot{H}^{1}_{x}}\rightarrow 0 \quad \text{as $n\rightarrow \infty$},
\end{split}
\end{equation}
Next we set $e:=-\lambda_{1}|\tilde{u}_{n}|^{2}\tilde{u}_{n}-\lambda_{2}(K\ast|\tilde{u}_{n}|^{2})\tilde{u}_{n}$. We deduce by 
interpolation 
\[\begin{split}
\|\nabla [(K\ast|\tilde{u}_{n}|^{2})\tilde{u}_{n}] \|_{L^{\frac{10}{7}}_{x}}
\lesssim \| (K\ast\nabla|\tilde{u}_{n}|^{2})\tilde{u}_{n}] \|_{L^{\frac{10}{7}}_{x}}
+\|(K\ast|\tilde{u}_{n}|^{2})\nabla\tilde{u}_{n} \|_{L^{\frac{10}{7}}_{x}}\\
\lesssim \|  \nabla |\tilde{u}_{n}|^{2} \|_{L^{\frac{5}{2}}_{x}}\| \tilde{u}_{n}  \|_{L^{\frac{10}{3}}_{x}}+
\||\tilde{u}_{n}|^{2} \|_{L^{\frac{5}{2}}_{x}}\|  \nabla \tilde{u}_{n}\|_{L^{\frac{10}{3}}_{x}}\\
\lesssim \|  \nabla \tilde{u}_{n}\|_{L^{\frac{10}{3}}_{x}}\| \tilde{u}_{n}  \|_{L^{10}_{x}}\| \tilde{u}_{n}\|_{L^{\frac{10}{3}}_{x}}.
\end{split}\]
Thus we infer that
\begin{equation}\label{Ker}
\|\nabla [(K\ast|\tilde{u}_{n}|^{2})\tilde{u}_{n}] \|_{L^{\frac{10}{7}}_{t,x}}\lesssim
\|  \nabla \tilde{u}_{n}\|_{L^{\frac{10}{3}}_{t,x}}\| \tilde{u}_{n}  \|_{L^{10}_{t,x}}\| \tilde{u}_{n}\|_{L^{\frac{10}{3}}_{t,x}}
\end{equation}
Similarly,
\begin{equation}\label{Ker22}
\|\nabla [|\tilde{u}_{n}|^{2}\tilde{u}_{n}] \|_{L^{\frac{10}{7}}_{t,x}}\lesssim
\|  \nabla \tilde{u}_{n}\|_{L^{\frac{10}{3}}_{t,x}}\| \tilde{u}_{n}  \|_{L^{10}_{t,x}}\| \tilde{u}_{n}\|_{L^{\frac{10}{3}}_{t,x}}.
\end{equation}
Combining \eqref{Ker}, \eqref{Ker22} and \eqref{Spli} we obtain
\begin{equation}\label{Inter1}
\begin{split}
\|\nabla   e \|_{L^{\frac{10}{7}}_{t, x}}\lesssim 
\|  \nabla \tilde{u}_{n}\|_{L^{\frac{10}{3}}_{t,x}}\| \tilde{u}_{n}  \|_{L^{10}_{t,x}}\| \tilde{u}_{n}\|_{L^{\frac{10}{3}}_{t,x}}
\lesssim C(\|\phi\|_{\dot{H}^{1}_{x}})\lambda^{1-\theta}_{n}\rightarrow0,
 \end{split}
\end{equation}
as $n\rightarrow\infty$. Thus, by \eqref{Spli}, \eqref{Inter1} and Lemma \ref{stabi} 
we see that for $n$ sufficiently large  there exists a unique global solution $u_{n}$ of \eqref{NLS} with initial data 
$u_{n}(0)=\phi_{n}$ such that \eqref{Csi} holds. 
Finally, \eqref{Small11} and \eqref{Small22}  can be proved along the same lines as Proposition 8.3 in \cite{KillipOhPoVi2017}.
\end{proof}

\begin{proof}[Proof of Proposition \ref{PSmc}] 
By translating $u_{n}$ in the time, we can have that $t_{n}\equiv 0$ in \eqref{Blowp2}.
First, by \eqref{Enl} we deduce that
\[\|u_{n}(0)\|^{2}_{H^{1}}  \lesssim \L(u_{n})  \lesssim \tau_{c}.
 \]
Using Theorem \ref{Profi} to $\left\{u_{n}(0)\right\}_{n\in \N}$ we write (after extracting a subsequence) for each
$J\leq J^{\ast}$,
\begin{equation}\label{Dpe}
u_{n}(0)=\sum^{J}_{j=1}\phi^{j}_{n}+W^{J}_{n}
\end{equation}
where the various sequences satisfy properties of Theorem \ref{Profi} and Lemma \ref{Lde}. Moreover, up to subsequence, we may assume that 
$M(u_{n})\rightarrow M_{0}$, $E(u_{n})\rightarrow E_{0}$ and therefore $\tau_{c}=\L(M_{0}, E_{0})$.
We define $\phi^{j}$ in the following form:
 \begin{itemize}
	\item For the case $\lambda^{j}_{n}\equiv 1$ and $t^{n}_{j}\rightarrow+\infty$ as $n\rightarrow\infty$, we define $\psi^{j}$  to be the global solution  of \eqref{NLS} which scatters to $e^{it\Delta}\phi^{j}$ when $t\rightarrow+\infty$. Similarly, if $\lambda^{j}_{n}\equiv 1$ and $t^{n}_{j}\rightarrow-\infty$ as $n\rightarrow\infty$, we define $\psi^{j}$  to be the global solution  of \eqref{NLS} which scatters to $e^{it\Delta}\phi^{j}$ when $t\rightarrow-\infty$. In either case, we define the global solution to \eqref{NLS}
\[\psi^{j}_{n}(t,x):=\psi^{j}(t+t^{j}_{n}, x-x^{j}_{n}).\]
\item For the case $\lambda^{j}_{n}\equiv 1$ and $t^{n}_{j}\equiv 0$, we define $\psi_{n}^{j}$ to be the global solution of \eqref{NLS}
	with the initial data $\psi_{n}^{j}(0)=\phi_{n}^{j}$.
	\item For the case $\lambda^{j}_{n}\rightarrow 0$ as $n\rightarrow\infty$, we define $\psi_{n}^{j}$ to be the global solution of \eqref{NLS}
	with the initial data $\psi_{n}^{j}(0)=\phi_{n}^{j}$  established in Lemma \ref{Lpsci}.
\end{itemize}
Therefore, associated to  $\phi_{n}^{j}$ we have a new nonlinear profile  $\psi_{n}^{j}(0)$ such that
	\begin{equation}\label{Apro11}
\|\psi_{n}^{j}(0)- \phi_{n}^{j}\|_{H^{1}_{x}}\rightarrow0,\quad \text{as $n\rightarrow\infty$}.
\end{equation}
 By the energy Pythagorean expansion \eqref{PE} and \eqref{PE11} we see that
\begin{align}\label{Pv1}
&\limsup_{n\rightarrow\infty}\sum^{J}_{j=1}M(\psi_{n}^{j})+M(W^{J}_{n})\leq M_{0},\\\label{Pv2}
&\limsup_{n\rightarrow\infty}\sum^{J}_{j=1}E(\psi_{n}^{j})+E(W^{J}_{n})\leq E_{0},
\end{align}
for any $J$. Since the profiles are non-trivial, from \eqref{Apro11} \eqref{Pv1} and inequality \eqref{Edp} implies that for $n$ big enough $E(\psi_{n}^{j})> 0$. And, again by \eqref{Pv1}, \eqref{Pv2} and \eqref{Edp} we obtain that  $E(W_{n}^{j})\geq 0$ for $n$ large.

From \eqref{Pv1} and \eqref{Pv2}, there are two scenarios to consider; we will show that Scenario 1 leads the compactness conclusion; on the other hand, we will show that Scenario 2  leads to contradiction and therefore does not occur.\\
 \textbf{Scenario 1:} \begin{equation}\label{IDS}
\sup_{j}\limsup_{n\rightarrow\infty}M(\psi_{n}^{j})=M_{0}\quad
\mbox {and} \quad \sup_{j}\limsup_{n\rightarrow\infty}E(\psi_{n}^{j})=E_{0}.
\end{equation}
Since $M(W^{J}_{n})\geq0$ and $E(W^{J}_{n})\geq0$, we have that $\limsup_{n\rightarrow\infty}E(\psi_{n}^{j})=E_{0}$ for some $j$.
We may have $j=1$ by reordering. Thus we see that $J=1$ and $W^{1}_{n}\rightarrow 0$ in $H_{x}^{1}$ as $n\rightarrow\infty$.
Indeed, since  $\limsup_{n\rightarrow\infty}E(W_{n}^{1})=0$ and $\limsup_{n\rightarrow\infty}M(W_{n}^{1})=0$, it follows from inequality \eqref{Enl} that $\limsup_{n\rightarrow\infty}\|W_{n}^{1}\|^{2}_{H^{1}}=0$. Therefore we have
\begin{equation}\label{Csc}
u_{n}(0)=\phi^{1}_{n}+W_{n}^{1}, \quad \lim_{n\rightarrow\infty}\|W_{n}^{1}\|^{2}_{H^{1}}=0.
\end{equation}
If $\lambda^{1}_{n}\equiv 0$ and $t^{1}_{n}\equiv 0$, then $u_{n}(0, \cdot+x^{1}_{n})\rightarrow \phi^{1}$ strongly in $H^{1}(\R^{3})$.
On the other hand, suppose that $\lambda^{1}_{n}\equiv 0$ and $t^{1}_{n}\rightarrow \infty$ as $n\rightarrow\infty$. By Sobolev embedding, Strichartz estimates, monotone convergence  theorem and \eqref{Csc} we infer that
\begin{equation}\label{fgc11}
\begin{split}
\| e^{i\Delta}u_{n}(0)  \|_{L^{10}_{t,x}[0,\infty)\times \R^{3}}\leq \| \phi^{1}_{n}  \|_{L^{10}_{t,x}[0,\infty)\times \R^{3}}+
\| W_{n}^{1} \|_{L^{10}_{t,x}[0,\infty)\times \R^{3}},\\
\lesssim
\| \phi^{1} \|_{L^{10}_{t,x}[t^{1}_{n},\infty)\times \R^{3}}+\|W_{n}^{1}\|_{H^{1}}\rightarrow0,
\end{split}
\end{equation}
as $n\rightarrow\infty$. Let $\epsilon>0$. We set $\tilde{u}:=e^{i\Delta}u_{n}(0)$ and $e:=-\lambda_{1}|\tilde{u}_{n}|^{2}\tilde{u}_{n}-\lambda_{2}(K\ast|\tilde{u}_{n}|^{2})\tilde{u}_{n}$. Then for $n$ sufficiently large, combining \eqref{Inter1} and \eqref{fgc11} we obtain
\[\| \tilde{u}  \|_{L^{10}_{t,x}[0,\infty)\times \R^{3}}+ \| \nabla e  \|_{L^{10}_{t,x}[0,\infty)\times \R^{3}}\leq \epsilon.\]
From \eqref{Csc}, we may apply the Lemma \ref{stabi} to obtain that for $n$ sufficiently large $\|u_{n} \|_{L^{10}_{t,x}([0,\infty)\times \R^{3})}$ is negligible, a contradiction with \eqref{Blowp2}. An analogous argument holds when  $\lambda^{1}_{n}\equiv 0$ and $t^{1}_{n}\rightarrow -\infty$ as $n\rightarrow\infty$ 
Finally, assume that $\lambda^{1}_{n}\rightarrow 0$  as $n\rightarrow\infty$. From  \eqref{Csc} we have
\begin{equation}\label{Apcv}
\|\psi_{n}^{1}(0)- u_{n}(0)\|_{H^{1}_{x}}\rightarrow0,\quad \text{as $n\rightarrow\infty$}.
\end{equation}
By Proposition \ref{Lpsci}, \eqref{Csi}, \eqref{Apcv} and Lemma \ref{stabi} (perturbation theory) we infer that 
$\| \nabla u_{n} \|_{S^{0}(\R)}$ is uniformly limited for large $n$, which is a contradiction with our hypothesis \eqref{Blowp2}.

\textbf{Scenario 2:} On the contrary,  we assume that  \eqref{IDS} fails for all $j$. Then there exists $\delta>0$ such that
\begin{equation}\label{deltaco}
\sup_{j}\limsup_{n\rightarrow\infty}M(\psi_{n}^{j})\leq M_{0}-\delta\quad
\mbox {or} \quad \sup_{j}\limsup_{n\rightarrow\infty}E(\psi_{n}^{j})\leq E_{0}-\delta.
\end{equation}
Now the idea of the proof is approximate
\[u_{n}(t)\approx \sum^{J}_{j=1}\psi^{j}_{n}(t)+e^{it\Delta}W^{J}_{n},\]
under tree cases: $\lambda^{j}_{n}\equiv 1$ and $t^{n}_{j}\rightarrow\pm\infty$ as $n\rightarrow\infty$; $\lambda^{j}_{n}\equiv 1$ and $t^{n}_{j}\equiv 0$; $\lambda^{j}_{n}\rightarrow 0$ as $n\rightarrow\infty$, and we use perturbation argument (Lemma \ref{stabi}) to obtain a contradiction with \eqref{Blowp2}.
 With this in mind, we define
\[u^{J}_{n}(t):=\sum^{J}_{j=1}\psi^{j}_{n}(t)+e^{it\Delta}W^{J}_{n}.
\]
Now consider the equation
\[ (i\partial_{t}+\Delta) {u}^{J}_{n}=\lambda_{1}|{u}^{J}_{n}|^{2}{u}^{J}_{n}+\lambda_{2}(K\ast|{u}^{J}_{n}|^{2}){u}^{J}_{n}+|{u}^{J}_{n}|^{4}{u}^{J}_{n}+e^{J}_{n},   \]
where the ``error'' $e^{J}_{n}$ is given by $e^{J}_{n}:=(i\partial_{t}+\Delta) {u}^{J}_{n}-\lambda_{1}|{u}^{J}_{n}|^{2}{u}^{J}_{n}-\lambda_{2}(K\ast|{u}^{J}_{n}|^{2}){u}^{J}_{n}+|{u}^{J}_{n}|^{4}{u}^{J}_{n}$.
We want to apply Lemma \ref{stabi}. We assume the following two claims for a moment to conclude the proof.\\
\textbf{Claim I.} We have the following global space bound
\begin{equation}\label{Gstb}
\sup_{J}\limsup_{n\rightarrow\infty}\left[ \| {u}^{J}_{n}  \|_{L^{10}_{t,x}(\R\times \R^{3})} 
+\| {u}^{J}_{n}  \|_{L_{t}^{\frac{10}{3}}H^{1,\frac{10}{3}}_{x}(\R\times \R^{3})} \right]\lesssim_{\tau_{c},\delta} 1.
\end{equation}
\textbf{Claim II.} Let $\epsilon>0$. For $J$ big enough (depending on $\epsilon$) we have
\begin{equation}\label{Eimpo}
\limsup_{n\rightarrow\infty}\|\nabla  e^{J}_{n}  \|_{L^{\frac{10}{7}}_{x}(\R\times \R^{3})}\leq\epsilon.
\end{equation}
Now from \eqref{Apro11} and \eqref{Dpe} we get
\begin{equation}\label{Aps}
\|u^{J}_{n}(0)-\psi_{n}^{j}(0)\|_{H^{1}_{x}}\rightarrow0,\quad \text{as $n\rightarrow\infty$}.
\end{equation}
Moreover, by interpolation inequality we obtain
\[
\begin{split}
\| \<\nabla\> (|u^{J}_{n}|^{4})u^{J}_{n}  \|_{L^{\frac{10}{7}}_{t,x}}&\lesssim \| u^{J}_{n}  \|^{4}_{L^{10}_{t,x}}\| \<\nabla\> u^{J}_{n}  \|_{L^{\frac{10}{3}}_{t,x}}\\
\| \<\nabla\> (|u^{J}_{n}|^{2})u^{J}_{n}  \|_{L^{\frac{10}{7}}_{t,x}}&
\lesssim \| u^{J}_{n}  \|_{L^{10}_{t,x}}\| u^{J}_{n}  \|_{L^{\frac{10}{3}}_{t,x}}\| \<\nabla\> u^{J}_{n}  \|_{L^{\frac{10}{3}}_{t,x}}\\
\| \<\nabla\> (K\ast|u^{J}_{n}|^{2})u^{J}_{n}  \|_{L^{\frac{10}{7}}_{t,x}}&
\lesssim \|u^{J}_{n}\|_{L^{10}_{t,x}}\| u^{J}_{n}  \|_{L^{\frac{10}{3}}_{t,x}}\| \<\nabla\> u^{J}_{n}  \|_{L^{\frac{10}{3}}_{t,x}},\\
\end{split}
\]
where all space-time norms will be taken over $\R\times\R^{3}$.
Thus, using the Duhamel formula and estimating via Stricharzt  we have
\begin{equation}\label{DS2}
\begin{split}
\| u^{J}_{n}  \|_{L_{t}^{\infty}L^{2}_{x}(\R\times\R^{3})}\lesssim 
\|u^{J}_{n}\|_{L^{10}_{t,x}(\R\times\R^{3})}\| u^{J}_{n}  \|_{L^{\frac{10}{3}}_{t,x}(\R\times\R^{3})}\| \<\nabla\> u^{J}_{n}  \|_{L^{\frac{10}{3}}_{x}(\R\times\R^{3})}\\
+\| u^{J}_{n}  \|^{4}_{L^{10}_{t,x}}\| \<\nabla\> u^{J}_{n}  \|_{L^{\frac{10}{3}}_{x}(\R\times\R^{3})}
+\|\nabla e^{J}_{n}\|_{L^{\frac{10}{7}}_{x}(\R\times\R^{3})}.
\end{split}
\end{equation}
Combining \eqref{Gstb}, \eqref{Eimpo} and \eqref{DS2} we have for $J$ sufficiently large
\begin{equation}\label{Fis2}
\limsup_{n\rightarrow\infty}\| u^{J}_{n}  \|_{L_{t}^{\infty}L^{2}_{x}(\R\times\R^{3})}\lesssim_{\tau_{c},\epsilon,\delta}1.
\end{equation}
Applying \eqref{Fis2}, \eqref{Gstb} , \eqref{Aps}, \eqref{Eimpo} to Lemma \ref{stabi} gives $\| u_{n}  \|_{L^{10}_{t,x}}\lesssim_{\tau_{c},\epsilon,\delta}1$ for  $n$ sufficiently large, which a contradiction with \eqref{Blowp2}.

Therefore, it remains to establish the above claims. Indeed, using Lemma \ref{Lel} (v) and \eqref{deltaco}, for each finite $J\leq J^{\ast}$, we see that for $n$ sufficiently large $M(\psi^{j}_{n})\leq M_{0}-\delta/2$ or $E(\psi^{j}_{n})\leq E_{0}-\delta/2$ and $1\leq j\leq J$. Thus
we get that there exists $\epsilon_{1}=\epsilon_{1}(\delta)>0$ such that $\L(\psi^{j}_{n})\leq \tau_{c}-\epsilon_{1}$ for $n$ sufficiently large. Then by definition of $\tau_{c}$ given by \eqref{tcri} we infer that
\begin{equation}\label{Udez}
\| \psi^{j}_{n}  \|_{L^{10}_{t,x}}\lesssim_{\delta, \tau_{c}} 1.
\end{equation}
Thus, from inequalities \eqref{Gbc}, \eqref{S1} and \eqref{Edp} we see that
\begin{equation}\label{Estrii}
\| \psi^{j}_{n}  \|_{L^{10}_{t,x}(\R\times\R^{3})}+\|\nabla  \psi^{j}_{n}\|_{L^{\frac{10}{3}}_{t,x}(\R\times\R^{3})}
\lesssim_{\delta, \tau_{c}} \|\nabla \psi^{j}_{n}(0)\|_{L^{2}_{t,x}(\R\times\R^{3})}\lesssim_{\delta, \tau_{c}} 
E(\psi^{j}_{n})^{1/2}
\end{equation}
and
\begin{equation}\label{masses}
\|  \psi^{j}_{n}\|_{L^{\frac{10}{3}}_{t,x}(\R\times\R^{3})}\lesssim_{\delta, \tau_{c}}
M(\psi^{j}_{n})^{1/2}.
\end{equation}
\begin{proof}[Proof of Claim I]  
Throughout the proof, all space-time norms will be taken over $\R\times\R^{3}$. We fix $J\leq J^{\ast}$. First, by \eqref{Estrii}, Lemma \ref{Lpsci}
and following the same argument developed in \cite[Lemma 9.2]{KillipOhPoVi2017} we have for $j\neq k$
\begin{equation}\label{Apliq}
\lim_{n\rightarrow\infty}[\|  \psi^{j}_{n}   \psi^{k}_{n}  \|_{L^{5}_{t,x}}+
\|  \psi^{j}_{n}  \nabla \psi^{k}_{n}  \|_{L^{\frac{5}{2}}_{t,x}}+\| \nabla \psi^{j}_{n}  \nabla \psi^{k}_{n}  \|_{L^{\frac{5}{3}}_{t,x}}
+\|  \psi^{j}_{n}   \psi^{k}_{n}  \|_{L^{\frac{5}{3}}_{t,x}}]=0.
\end{equation}
Moreover, we also have for $j\neq k$
\begin{equation}\label{Cs2}
\lim_{n\rightarrow\infty}\|  \psi^{j}_{n}   \psi^{k}_{n}  \|_{L^{\frac{5}{2}}_{t,x}}=0.
\end{equation}
Indeed, if $\lambda^{j}_{n}\equiv 1$ and  $\lambda^{k}_{n}\equiv 1$, then using the same argument as in Lemma 9.2 in 
\cite{KillipOhPoVi2017} again, yields \eqref{Cs2}. Now, in the case when $\lambda^{k}_{n}\rightarrow 0$, by \eqref{Apro11}, \eqref{masses}, Bernstein inequality and definition of 
$\phi_{n}^{j} $ given by \eqref{fucti} we get
\[ \|  \psi^{k}_{n}  \|_{L^{\frac{10}{3}}_{t,x}}\lesssim_{\delta, \tau_{c}} \|  \psi^{k}_{n} (0)  \|_{L^{2}_{x}} 
\lesssim_{\delta, \tau_{c}} (\lambda^{k}_{n})^{1-\theta}\|\nabla \phi^{j}  \|_{L^{2}_{x}}+\| \psi^{k}_{n} (0)- \phi^{k}_{n} \|_{L^{}_{x}}
\rightarrow0,
 \]
as $n\rightarrow\infty$. And analogous argument holds for $\lambda^{j}_{n}\rightarrow 0$. Using \eqref{Udez}, inequality above and interpolation  we obtain
\[ 
\|  \psi^{j}_{n}   \psi^{k}_{n}  \|_{L^{\frac{5}{2}}_{t,x}}\lesssim \|  \psi^{j}_{n}  \|_{L^{10}_{t,x}}
\| \psi^{k}_{n}  \|_{L^{\frac{10}{3}}_{t,x}}\rightarrow0\quad  \mbox{as $n\rightarrow\infty$},
\]
which yields \eqref{Cs2}. 

From \eqref{Apliq},\eqref{masses}, \eqref{Pv1} and Stricharzt estimates we get
\begin{equation}\label{Cla11}
\begin{split}
\| u^{J}_{n}  \|^{2}_{L^{\frac{10}{3}}_{t,x}}\lesssim \sum^{J}_{j=1}\| \psi^{j}_{n}  \|^{2}_{L^{\frac{10}{3}}_{t,x}}+
\sum_{j\neq k}\| \psi^{j}_{n} \psi^{k}_{n}  \|^{2}_{L^{\frac{5}{3}}_{t,x}}+
\| W^{J}_{n} \|^{2}_{L^{2}_{x}}\\
\lesssim_{\delta, \tau_{c}}\sum^{J}_{j=1}M( \psi^{j}_{n} )+\sum_{j\neq k}o(1)+M( W^{J}_{n})\lesssim_{\delta, \tau_{c}} 1+J^{2}o(1),
\end{split}
\end{equation}
as $n\rightarrow\infty$. Similarly, by \eqref{Apliq},\eqref{Estrii}, \eqref{Edp} and \eqref{Pv2}  we obtain
\begin{equation}\label{Cla22}
\begin{split}
\| u^{J}_{n}  \|^{2}_{L^{10}_{t,x}}+\| u^{J}_{n}  \|^{2}_{L^{\frac{10}{3}}_{x}}\lesssim_{\delta, \tau_{c}}\sum^{J}_{j=1}E( \psi^{j}_{n} )+\sum_{j\neq k}o(1)+E( W^{J}_{n})\\
\lesssim_{\delta, \tau_{c}} 1+J^{2}o(1), \quad \mbox{as $n\rightarrow\infty$}.
\end{split}
\end{equation}
Combining \eqref{Cla11} and \eqref{Cla22} yields \eqref{Gstb}.
\end{proof}
\begin{proof}[Proof of Claim II]
Since $\psi^{j}_{n}$ is solution of \eqref{NLS} we have
\begin{align}\label{Fd1}
e^{J}_{n}&=\sum^{J}_{j=1}F(\psi^{j}_{n}) -F(\sum^{J}_{j=1}\psi^{j}_{n})\\\label{Fd22}
&=F(u^{J}_{n}-e^{it\Delta }W^{J}_{n})-F( u^{J}_{n} ),
\end{align}
where $F(u)=F_{1}(u)+F_{2}(u)$, with $F_{1}(u)=\lambda_{2}(K\ast|u|^{2})u$ and $F_{2}(u)=\lambda_{1}|u|^{2}u+|u|^{4}u$.
By interpolation we have,
\begin{equation}\label{F1K}
\begin{split}
\|\sum^{J}_{j=1}F_{1}(\psi^{j}_{n}) -F_{1}(\sum^{J}_{j=1}\psi^{j}_{n})\|_{L^{\frac{10}{7}}_{t,x}}
\lesssim_{J} 
\sum_{j\neq k}\|[K\ast\nabla(\psi^{j}_{n}\overline{\psi}^{k}_{n})]\psi^{j}_{n}  \|_{L^{\frac{10}{7}}_{t,x}}\\
+\sum_{j\neq k}\| [K\ast(\psi^{j}_{n}\overline{\psi}^{k}_{n})]\nabla\psi^{j}_{n}  \|_{L^{\frac{10}{7}}_{t,x}}\\
\lesssim_{J}\sum_{j\neq k}[ \|  \psi^{j}_{n}  \|_{L^{\frac{10}{3}}_{t,x}}\|  \psi^{k}_{n} \nabla \psi^{k}_{n} \|_{L^{\frac{5}{2}}_{x}}+
\| \nabla \psi^{j}_{n}  \|_{L^{\frac{10}{3}}_{t,x}}\|  \psi^{j}_{n} \psi^{k}_{n}  \|_{L^{\frac{5}{2}}_{t,x}}].
\end{split}
\end{equation}
Similarly,
\begin{equation}\label{F122K}
\begin{split}
\|\sum^{J}_{j=1}F_{2}(\psi^{j}_{n})-F_{2}(\sum^{J}_{j=1}\psi^{j}_{n})\|_{L^{\frac{10}{7}}_{t,x}}
\lesssim_{J} 
\sum_{j\neq k}[\|  \psi^{k}_{n}  \|^{3}_{L^{10}_{t,x}}\|  \psi^{j}_{n} \nabla\psi^{k}_{n}  \|_{L^{\frac{5}{2}}_{t,x}}\\
+
\|  \psi^{j}_{n}  \|^{3}_{L^{10}_{t,x}}\|  \psi^{k}_{n} \nabla\psi^{j}_{n}  \|_{L^{\frac{5}{2}}_{t,x}}+
\|  \psi^{k}_{n}  \|_{L^{\frac{10}{3}}_{t,x}}\|  \psi^{j}_{n} \nabla\psi^{k}_{n}  \|_{L^{\frac{5}{2}}_{t,x}}+
\|  \psi^{j}_{n}  \|_{L^{\frac{10}{3}}_{t,x}}\|  \psi^{k}_{n} \nabla\psi^{j}_{n}  \|_{L^{\frac{5}{2}}_{t,x}}].
\end{split}
\end{equation}
Combining \eqref{Estrii}, \eqref{masses}, \eqref{Apliq}, \eqref{F1K}, \eqref{Cs2} and \eqref{F122K} we have
\begin{equation}\label{Clm}
\limsup_{n\rightarrow\infty}
\| \eqref{Fd1} \|_{L^{\frac{10}{7}}_{t, x}}\lesssim_{J,\delta,\tau_{c}}o(1)\quad, \mbox{as $n\rightarrow\infty$}.
\end{equation}
On the other hand, by interpolation
\begin{equation}\label{F2L}
\begin{split}
\| \eqref{Fd22}  \|_{L^{\frac{10}{7}}_{x}}
\lesssim
\|  u^{J}_{n} \|_{L^{\frac{10}{3}}_{x}}\| u^{J}_{n} \nabla e^{it\Delta}W^{J}_{n} \|_{L^{}_{x}}+
\| \nabla u^{J}_{n} \|_{L^{\frac{10}{3}}_{x}}\|  e^{it\Delta}W^{J}_{n}  \|_{L^{10}_{x}}\|  u^{J}_{n} \|_{L^{\frac{10}{3}}_{x}}+
\\
\|  u^{J}_{n} \|_{L^{\frac{10}{3}}_{x}}\|  e^{it\Delta}W^{J}_{n}  \|_{L^{10}_{x}}\|\nabla  e^{it\Delta}W^{J}_{n}  \|_{L^{\frac{10}{3}}_{x}}+
\|  e^{it\Delta}W^{J}_{n} \|_{L^{\frac{10}{3}}_{x}}\|  e^{it\Delta}W^{J}_{n} \|_{L^{10}_{x}}\|\nabla  e^{it\Delta}W^{J}_{n}  \|_{L^{\frac{10}{3}}_{x}}+
\\
\| \nabla u^{J}_{n} \|_{L^{\frac{10}{3}}_{x}}\|  e^{it\Delta}W^{J}_{n}  \|_{L^{10}_{x}}\|u^{J}_{n}  \|^{3}_{L^{10}_{x}}
\| \nabla u^{J}_{n} \|_{L^{\frac{10}{3}}_{x}}\|  e^{it\Delta}W^{J}_{n}  \|_{L^{10}_{x}}\|
 e^{it\Delta}W^{J}_{n} \|^{3}_{L^{10}_{x}}\\
+\| K\ast|u^{J}_{n}|^{2} \nabla e^{it\Delta}W^{J}_{n} \|_{L^{\frac{10}{7}}_{x}}+
\|  u^{J}_{n} \|^{3}_{L^{\frac{10}{3}}_{x}}\| u^{J}_{n}  e^{it\Delta}W^{J}_{n} \|_{L^{\frac{5}{2}}_{x}}+
\\
\|  e^{it\Delta}W^{J}_{n}  \|^{4}_{L^{10}_{x}}\| \nabla u^{J}_{n} \|_{L^{\frac{10}{3}}_{x}}+
\|  e^{it\Delta}W^{J}_{n}  \|^{4}_{L^{10}_{x}}\| \nabla e^{it\Delta}W^{J}_{n} \|_{L^{\frac{10}{3}}_{x}}.
\end{split}
\end{equation}
Notice that 
\begin{align}\label{Ca1}
&\lim_{J\rightarrow J^{\ast}}\limsup_{n\rightarrow\infty}\|u^{J}_{n} \nabla e^{it\Delta}W^{J}_{n}\|_{L^{^{\frac{5}{2}}}_{x}}
=0.
\end{align}
Such statement can be proved along the same lines as Lemma 9.5 in \cite{KillipOhPoVi2017}. Now we show that
\begin{align}\label{SUPK}
&\lim_{J\rightarrow J^{\ast}}\limsup_{n\rightarrow\infty}\| K\ast|u^{J}_{n}|^{2} \nabla e^{it\Delta}W^{J}_{n} \|_{L^{\frac{10}{7}}_{x}}=0.
\end{align}
By orthogonality we have that for $j\not=k$, 
\begin{align*}
  \| K\ast\(\psi^j_{n}\bar \psi^k_n\) \nabla e^{it\Delta}W^{J}_{n}
  \|_{L^{\frac{10}{7}}_{t,x}}&\le \| K\ast\(\psi^j_{n}\bar \psi^k_n\)
  \|_{L^{\frac{5}{2}}_{t,x}} \|\nabla e^{it\Delta}W^{J}_{n}
                               \|_{L^{\frac{10}{3}}_{t,x}}\\
  &\lesssim \|\psi^j_{n}\bar \psi^k_n
  \|_{L^{\frac{5}{2}}_{t,x}} \|\nabla e^{it\Delta}W^{J}_{n}
  \|_{L^{\frac{10}{3}}_{t,x}}\rightarrow 0,
\end{align*}
as $n\to\infty$, where we have used Lemma~\ref{LLk} and \eqref{Cs2}. This implies by triangle inequality, H\"older inequality, \eqref{Sr}, \eqref{masses} and \eqref{Pv1}
\[\begin{split}
\lim_{J\rightarrow J^{\ast}}\limsup_{n\rightarrow\infty}\| K\ast|u^{J}_{n}|^{2} \nabla e^{it\Delta}W^{J}_{n} \|_{L^{\frac{10}{7}}_{x}}
\lesssim
\lim_{J\rightarrow J^{\ast}}\limsup_{n\rightarrow\infty}\|  K\ast\Big(\sum^{J}_{j=1}|\psi^j_{n}|^{2}\Big)\nabla e^{it\Delta}W^{J}_{n} \|_{L^{\frac{10}{7}}_{x}}\\
+\lim_{J\rightarrow J^{\ast}}\limsup_{n\rightarrow\infty}\sum^{J}_{j=1}\| \psi^j_{n}  \|_{L_{t,x}^{\frac{10}{3}}}\|  e^{it\Delta}W^{J}_{n}  \|_{L^{10}_{x}}
\| \nabla e^{it\Delta}W^{J}_{n}  \|_{L^{\frac{10}{3}}_{x}}\\
+\lim_{J\rightarrow J^{\ast}}\limsup_{n\rightarrow\infty}\|  e^{it\Delta}W^{J}_{n}  \|_{L^{10}_{x}}\|  e^{it\Delta}W^{J}_{n}  \|_{L^{\frac{10}{3}}_{x}}
\| \nabla e^{it\Delta}W^{J}_{n}  \|_{L^{\frac{10}{3}}_{x}}\\
\lesssim
\lim_{J\rightarrow J^{\ast}}\limsup_{n\rightarrow\infty}\|  K\ast\Big(\sum^{J}_{j=1}|\psi^j_{n}|^{2}\Big)\nabla e^{it\Delta}W^{J}_{n} \|_{L^{\frac{10}{7}}_{x}}.
\end{split}\]
On the other hand, by H\"older inequality, \eqref{Cs2} \eqref{masses}, \eqref{Pv1} and \eqref{Estrii} we get
\[\begin{split}
\| \sum^{J}_{j=J^{'}} |\psi^j_{n}|^{2} \|^{2}_{L^{\frac{5}{2}}_{t,x}}\lesssim
\sum_{j\neq k}\| \psi^j_{n}\psi^k_{n}\|^{2}_{L^{\frac{5}{2}}_{x}}+
 \sum^{J}_{j=J^{'}}\| |\psi^j_{n}|^{2}\|^{2}_{L^{\frac{5}{2}}_{x}}\\
\lesssim \sum_{j\neq k}o(1)+ M_{0}\sum^{J}_{j=J^{'}}E(\psi^j_{n}), \quad \mbox{as $n\to \infty$}.
\end{split}\]
Thus, by \eqref{Pv2}  we infer that for any $\epsilon_{1}>0$ there exists $J^{'}=J^{'}(\epsilon_{1})\in \N$ such that
\[ \limsup_{n\rightarrow\infty}\| \sum^{J}_{j=J^{'}} |\psi^j_{n}|^{2} \|_{L^{\frac{5}{2}}_{t,x}}\leq  \epsilon_{1},\]
for all $J\leq J^{\ast}$. Then, from H\"older inequality, Lemma~\ref{LLk} and \eqref{Pv2} we get
\[
\lim_{J\rightarrow J^{\ast}}\limsup_{n\rightarrow\infty}\|  K\ast\Big(\sum^{J}_{j=J^{'}}|\psi^j_{n}|^{2}\Big)\nabla e^{it\Delta}W^{J}_{n} \|_{L^{\frac{10}{7}}_{x}}\lesssim \lim_{J\rightarrow J^{\ast}}\limsup_{n\rightarrow\infty}\| \sum^{J}_{j=J^{'}} |\psi^j_{n}|^{2} \|_{L^{\frac{5}{2}}_{t,x}} \lesssim \epsilon_{1}.
\]
Therefore, to prove \eqref{SUPK} it is enough to show that  
\[\lim_{J\rightarrow J^{\ast}}\limsup_{n\rightarrow\infty}\| K\ast|\psi^j_{n}|^{2} \nabla e^{it\Delta}W^{J}_{n}
  \|_{L^{\frac{10}{7}}_{t,x}}=0, \quad \text{for all $j\in \left\{1,2,\ldots, J^{'}\right\}$.}
\]
First we consider the case 	$\lambda_n^j\to 0$ as $n\to
\infty$. Thus we fix $1\leq j\leq J^{'}$ (notice that $J^{'}$ is
finite) with $\lambda_n^j\to 0$ as $n\to \infty$. Proceeding like in \cite{KillipOhPoVi2017}, by Lemma~\ref{Lpsci} we infer that for $\epsilon>0$ sufficiently small there exists $\zeta^{j}_{\epsilon}\in C^{\infty}_{c}(\R\times\R^{3})$ such that
\begin{equation}\label{phyap-bis}
\left\|\psi^j_{n}-\frac{1}{(\lambda^{j}_{n})^{{1/2}}}\zeta^{j}_{\epsilon}\Big(
  \frac{t}{(\lambda_n^j)^2} +t_n^j,\frac{x-x_n^j}{\lambda_n^j}\Big)
\right\|_{L^{10}_{t,x}}<\epsilon. 
\end{equation}
Using the fact that $K$ is homogeneous of degree $-3$, from
Lemma~\ref{LLk} we see that there exists $C$ such that 
\begin{equation}\label{deps-bis}
  \left\|K\ast|\psi^j_{n}|^{2} - \frac{1}{\lambda_n^j} K\ast
  |\zeta^{j}_{\epsilon}|^2 \Big( \frac{t}{(\lambda_n^j)^2}
  +t_n^j,\frac{x-x_n^j}{\lambda_n^j}\Big)\right\|_{L^{5}_{t,x}}\le C\epsilon.
\end{equation}
 Now the function $K\ast|\zeta^{j}_{\epsilon}|^2$ is compactly supported in time, and in $L^p_x(\R^{3})$ for all $p\in (1,\infty)$. Therefore, 
	there exists $\chi^j_\epsilon $ compactly supported in space-time such that
  \begin{equation}\label{Koz-bis}
	\| K\ast|\zeta^{j}_{\epsilon}|^2-\chi^j_\epsilon \|_{L^{5}_{t,x}}<\epsilon.
  \end{equation}
  In particular,
 \begin{equation}\label{newq-bis}
  \left\|K\ast|\psi^j_{n}|^{2} -
    (\lambda_n^j)^{-1}\chi^j_\epsilon \Big(
    \frac{t}{(\lambda_n^j)^2}
    +t_n^j,\frac{x-x_n^j}{\lambda_n^j}\Big)\right\|_{L^{5}_{t,x}}\lesssim
  \epsilon, 
\end{equation}
for $n$ sufficiently large. We set 
$\tilde{W}^{J}_{n}(t,x):=(\lambda^{j}_{n})^{\frac{1}{2}}[e^{i\Delta t}W^{J}_{n}]( (\lambda^{j}_{n})^{2}(t-t^{j}_{n}), \lambda^{j}_{n}x+x^{j}_{n})$.
Applying \cite[Lemma 1.8]{KillipOhPoVi2017} and commuting the dilation with the propagator, via H\"older inequality we obtain
\begin{equation}\label{clpe}
\|K\ast|\psi^j_{n}|^{2} \nabla e^{it\Delta}W^{J}_{n} \|_{L^{\frac{10}{7}}_{t,x}}
\lesssim \epsilon \| \nabla e^{it\Delta}W^{J}_{n}  \|_{L^{\frac{10}{3}}_{t,x}}
 + \|\chi^j_\epsilon  \nabla \tilde{W}^{J}_{n}\|_{L^{\frac{10}{7}}_{t,x}}.
\end{equation}
For the last term, H\"older inequality yields
\begin{equation}\label{fgc}
  \|\chi^j_\epsilon  \nabla
  \tilde{W}^{J}_{n}\|_{L^{\frac{10}{7}}_{t,x}}\le \|\chi^j_\epsilon
 \|_{L^{\frac{10}{3}}_{t,x}}\|
  \nabla \tilde{W}^{J}_{n}\|_{L^{\frac{10}{3}}_{t,x}}^{1/2}
  \| \nabla
  \tilde{W}^{J}_{n}\|_{L^{2}_{t,x}(\operatorname{supp} \chi^j_\epsilon)}^{1/2}.
\end{equation}
The quantity $\|\chi^j_\epsilon
 \|_{L^{\frac{10}{3}}_{t,x}}$ is finite since $\chi^j_\epsilon$ is
 compactly supported in space-time, with an implicit dependence upon
 $\epsilon$. 
Finally, since $\epsilon>0$ is arbitrary, by \eqref{clpe}-\eqref{fgc} and \eqref{Sr} we obtain the claim \eqref{SUPK}.
The other case $\lambda^{j}_{n}\equiv 1$ can be dealt with similarity.

Combining \eqref{Sr}, \eqref{Gstb}, \eqref{F2L} and  \eqref{Ca1}-\eqref{SUPK} we have
\begin{equation}\label{Vbn}
\lim_{J\rightarrow J^{\ast}}\limsup_{n\rightarrow\infty}\| \eqref{Fd22}  \|_{L^{\frac{10}{7}}_{x}}=0.
\end{equation}
Finally, \eqref{Eimpo} follows by \eqref{Clm} and \eqref{Vbn}.
\end{proof}

Thus completes the proof of Proposition \ref{PSmc}.
\end{proof}

\section{Extinction of the critical element}\label{Su8}
In this section we will apply the localized virial identities to preclude the critical element $u_{c}$ given in Theorem \ref{ECEC}.
We begin with the following result.

\begin{proposition}\label{BuIo}
Let $u_{c}(t)$ be the critical solution in Theorem \ref{ECEC}.
The following properties hold:\\
{\rm (i)(Precompacness of flow)} There exists a continuous path $x(t)$ in $\R^{3}$ such that $\left\{u_{c}(\cdot+x(t))\right\}$ is precompact
in $H^{1}(\R^{3})$.\\
{\rm(ii) (Uniform localization)} For every $\epsilon>0$ there exists $C(\epsilon)>0$ such that
\begin{equation}\label{Unflo}
\sup_{t\in \R}\int_{|x-x(t)|>C(\epsilon)}|\nabla u(x,t)|^{2}+| u(x,t)|^{2}+| u(x,t)|^{4}+| u(x,t)|^{6}dx\leq \epsilon.
\end{equation}
{\rm (iii)(Zero momentum of $u_{c}$)} The conserved momentum $P(u_{c}(t))=\int_{\R^{3}}2\IM(\overline{u}_{c}(t)\nabla u_{c}(t))dx$ is zero.\\
{\rm (iv)(Control of the spatial translation $x(t)$)} The spacial center function $x(t)$ in Item (i) satisfies
\[\left|\frac{x(t)}{t}\right|\rightarrow 0, \quad \text{as $t\rightarrow\pm\infty$}.\]
{\rm(v)(Uniform bound for $I(u_{c}(t))$)} There exists $\eta>0$ such that
\begin{equation}\label{ubv}
\begin{split}
I(u_{c}(t))=\|\nabla u_{c}(t)\|^{2}_{L^{2}}+\|u_{c}(t)\|^{6}_{L^{6}}+\frac{3}{4}\lambda_{1}\|u_{c}(t)\|^{4}_{L^{4}}+
\\
+\frac{3}{4}\lambda_{2}\|(K\ast|u_{c}(t)|^{2})|u_{c}(t)|^{2}\|_{L^{1}}\geq\eta, \quad \mbox{for all $t\in \R$}.
\end{split}
\end{equation}
\end{proposition}
\begin{proof}
The proof of Item (i) follows from exactly the same argument in \cite[Proposition 3.2]{DuyHolmerRoude2008}. 
By Item (i), Gagliardo-Nirenberg inequality and Sobolev embedding, property (ii) follows easily.
Now assume by contradiction  that $P(u_{c})\neq 0$. We consider the following global solution to \eqref{NLS},
\[w_{c}(t,x)=e^{ix\cdot\xi_{0}}e^{-it|\xi_{0}|^{2}}u_{c}(t, x-2\xi_{0}t),\]
where $\xi_{0}=-P(u_{c})/M(u_{c})$. Notice that $M(w_{c})=M(v_{c})$ and
\[E(w_{c})=E(u_{c})-\frac{1}{2}\frac{|P(u_{c})|^{2}}{M(u_{c})}<E(u_{c}).\]
This implies that $\L(w_{c})<\L(u_{c})=\tau_{c}$. Moreover, Theorem \ref{ECEC} implies
\begin{align*}
\|  w_{c} \|_{L^{10}_{t,x}([0,\infty)\times\R^{3})}&=\|  u_{c} \|_{L^{10}_{t,x}([0,\infty)\times\R^{3})}=\infty,\\
\|  w_{c} \|_{L^{10}_{t,x}((-\infty,0]\times\R^{3})}&=\|u_{c} \|_{L^{10}_{t,x}((-\infty,0]\times\R^{3})}=\infty,
\end{align*}
which is a contradiction with the definition of $u_{c}$. Next, applying \eqref{Unflo} and Item (iii), Item (iv) can be proved along the same lines as Proposition 10.2 in \cite{KillipOhPoVi2017}. We now prove \eqref{ubv}. We argue by contradiction. Suppose that \eqref{ubv} is false. Then there exists 
a sequence of times $\left\{t_{n}\right\}_{n\in \N}$ such that $\lim_{n\rightarrow\infty}I(u(t_{n}))=0$. Since $\left\{u(t_{n})\right\}_{n\in \N}$
is precompact modulo translation, there exists a sequence (still denoted by itself) and $v\in H^{1}(\R^{3})$ such that
$u(t_{n})\rightarrow v$ in $H^{1}(\R^{3})$ as $n\rightarrow \infty$. By strong convergence and continuity of $\L$ and $I$ we have
\[ I(v)=0 \quad \mbox{and} \quad \L(v)=\L(u(t))=\tau_{c}\in (0,\infty), \]
which is a contradiction with Lemma \ref{Lel}(ii). This completes the proof of proposition.
\end{proof}

Now we give the proof of Theorem \ref{TheS}.

\begin{proof}[{Proof of Theorem \ref{TheS}}]
We will show that the critical solution $u_{c}$ constructed in Theorem \ref{ECEC} can not exist. Indeed, consider the
localized Virial identity
\[V(t):=\int_{\R^{3}}\phi(x)|u(t,x)|^{2}dx.    \]
It follows from straightforward calculations that
\begin{align*}
	V^{\prime}(t)&=2\IM \int_{\R^{3}}\nabla\phi\cdot\nabla u \bar{u}dx             
\end{align*}
	and 
\[
\begin{split}
  V^{\prime\prime}(t)=4\int_{\R^{3}}[\RE\nabla \bar{u}\nabla^{2}\phi\nabla udx+\lambda_{1}\Delta\phi|u|^{4}+\frac{4}{3}\Delta\phi|u|^{6}  -\Delta^{2}\phi|u|^{2}]dx\\
	-2\int_{\R^{3}}\lambda_{2}\nabla\phi\nabla[ K\ast|u|^{2} ]|u|^{2}dx.
\end{split}
\]
If $\phi$ is a radial function we have
\[
\begin{split}
  V^{\prime\prime}(t)=4\int_{\R^{3}}\frac{\phi^{'}}{r}|\nabla u|^{2}dx+4\int_{\R^{3}}\(\frac{\phi^{''}}{r^{2}}-\frac{\phi^{'}}{r^{3}}\)|x\cdot\nabla u|^{2}dx\\
	+\int_{\R^{3}}\(\phi^{''}(r)+\frac{2}{r}\phi^{'}(r)\)(\lambda_{1}|u|^{4}+\frac{4}{3}|u|^{6})dx\\
	-\int_{\R^{3}}\Delta^{2}\phi|u|^{2}dx-2\lambda_{2}\int_{\R^{3}}\frac{\phi^{'}}{r}x\cdot\nabla[ K\ast|u|^{2} ]|u|^{2}dx.
\end{split}
\]
Now we choose the radial function $\phi(x)=R^{2}\psi(\frac{|x|}{R})$, where the function $\psi$ satisfies
\[\psi(r)=
\begin{cases} 
r^{2}, \quad 0\leq r\leq R;\\
0, \quad r\geq 2R,
\end{cases}  
\quad
0\leq\psi\leq r^{2}, \quad \psi^{''}\leq 2,\quad \psi^{4}\leq\frac{4}{R^{2}}.
\]
Then, using the estimate obtained in \cite[Lemma 6.2]{BellazziniForcella2019}, namely,
\[\begin{split}
-2\lambda_{2}\int_{\R^{3}}\frac{\phi^{'}}{r}x\cdot\nabla[ K\ast|u|^{2} ]|u|^{2}dx\geq6\lambda_{2}\int_{\R^{3}}[ K\ast|u|^{2} ]|u|^{2}dx\\
+
O\(\int_{|x|\geq R}[|\nabla u|^{2}+R^{-2}|u|^{2} +|u|^{4}](t,x)dx \),
\end{split}
\]
by the property of $\phi$ we get
\begin{equation}\label{VRCf}
\begin{split}
V^{\prime\prime}(t)\geq 8I(u(t))+O\(\int_{|x|\geq R}[|\nabla u|^{2}+|u|^{2} +|u|^{4}+|u|^{6}](t,x)dx \).
\end{split}
\end{equation}
We recall that from \eqref{ubv} there exists $\eta>0$ independent of $t$ such that $I(u(t))\geq \eta$.
Let $\epsilon>0$ be a parameter to be chosen sufficiently small below (depending on $\eta$). We deduce from Proposition \ref{BuIo} (iv) that there exists 
$t_{0}>0$ such that
\begin{equation}\label{Cpt}
|x(t)|\leq \epsilon t, \quad \text{for all $T\geq t_{0}$}.
\end{equation}
Given $T>t_{0}$, we put
\[R_{T}=C(\epsilon)+\epsilon T,\]
where $C(\epsilon)$ is as in \eqref{Unflo}. Then $\left\{|x|\geq R_{T}\right\}\subset \left\{|x-x(t)|\geq C(\epsilon)\right\}$ for all
$t\in [t_{0},T]$ and so, by \eqref{VRCf} and \eqref{Unflo}, we infer that for $\epsilon$ small enough (depending on $\eta$),
\begin{equation}\label{ufcv}
V^{\prime\prime}(t)\geq 2\eta, \quad \mbox{for all $t\in [t_{0},T]$}.
\end{equation}
Notice that from \eqref{Enl} we have
\begin{equation}\label{Limderi}
|V^{\prime}(t)|\leq CR\|u(t)   \|_{L^{2}_{x}}\| \nabla  u(t) \|_{L^{2}_{x}}\leq AR,
\end{equation}
for some constant $A$ independent of $t$ and $R>0$. Then integrating \eqref{ufcv} and applying \eqref{Limderi} we get
\[2\eta(T-t_{0})\leq \int^{T}_{t_{0}}V^{\prime\prime}(t)dt\leq |V^{\prime}(T)-V^{\prime}(t_{0})|
\leq2AR_{T}=2A\(C(\epsilon)+ \epsilon T\). \]
Choosing $\epsilon$ sufficiently small and taking $T$ large enough, we obtain a contradiction, 
the proof of Theorem \ref{TheS} is now completed.
\end{proof}

\section*{Acknowledgment}
The  author thanks R. Carles for his help in proving Proposition \ref{PSmc} and for useful comments
and suggestions that improved the paper.

\bibliographystyle{siam}
\bibliography{bibliografia}

\end{document}